\documentclass[aap,preprint]{imsart_modified}
\usepackage{fancyhdr}
\usepackage{latexsym}
\usepackage{graphics}
\usepackage{amsmath}
\usepackage{xspace}
\usepackage{amssymb}
\usepackage{psfrag}
\usepackage{epsfig}
\usepackage{fullpage}
\usepackage{amsmath,amsthm}
\usepackage{epsfig}
\usepackage{pst-all}

\usepackage{enumerate}

\newcommand {\bel}[1]{\begin{align*}}
\newcommand {\eel}[1]{\end{align*}}
\newcommand {\bea}{\begin{eqnarray}}
\newcommand {\eea}{\end{eqnarray}}

\newcommand{\pr}{\mathbb{P}}

\newcommand{\mb}[1]{\mbox{\boldmath $#1$}}

\newcommand{\ignore}[1]{\relax}

\def\bbr{{\mathbb R}}

\newcommand{\reals}{{\mathbb R}}

\newtheorem{theorem}{Theorem}

\newtheorem{lemma}{Lemma}

\newtheorem{corollary}{Corollary}

\newtheorem{definition}{Definition}
\newtheorem{question}{Question}
\newtheorem{observation}{Observation}

\everymath{\displaystyle}

\begin{document}
\begin{frontmatter}
\title{Second-order Markov random fields for independent sets on the infinite Cayley tree}
\runtitle{Second-order Markov random fields for independent sets on the infinite Cayley tree}
\begin{aug}
\author{\fnms{David\ A.} \snm{Goldberg}\ead[label=e2]{dgoldberg9@isye.gatech.edu}}
\affiliation{Georgia Institute of Technology}
  \address{
	Georgia Institute of Technology, Atlanta, GA, 30332\\	
	  \ \ \ \ \ \printead{e2}
}

\end{aug}
\begin{abstract}
Recently, there has been significant interest in understanding the properties of Markov random fields (M.r.f.) defined on on the independent sets of sparse graphs.  When these M.r.f. are restricted to pairwise interactions (i.e. hardcore model), much progress has been made.  However, considerably less is known in the presence of higher-order interactions, which arise e.g. in the analysis of independent sets with special properties and the study of resource-constrained communication networks.  In this paper, we further our understanding of such models by analyzing M.r.f. with second-order interactions on the independent sets of the infinite Cayley tree.  We prove that the associated Gibbsian specification satisfies the celebrated FKG Inequality whenever the local potentials defining the Hamiltonian satisfy a log-convexity condition.  Under this condition, we give necessary and sufficient conditions for the existence of a unique infinite-volume Gibbs measure in terms of an explicit system of equations, prove the existence of a phase transition, and give explicit bounds on the associated critical activity, which we prove to exhibit a certain robustness.  For potentials which are small perturbations of those coinciding to the hardcore model at the critical activity, we characterize whether the resulting specification has a unique infinite-volume Gibbs measure in terms of whether these perturbations satisfy an explicit linear inequality.  Our analysis reveals an interesting non-monotonicity with regards to biasing towards excluded nodes with no included neighbors. \end{abstract}

\end{frontmatter}
\section{Introduction}\label{introsec}
Recently, there has been a significant interest in combining ideas from probability, computer science, physics, statistics, and operations research, to shed light on the structure and complexity of combinatorial optimization, counting, and sampling problems (cf. \cite{mezard2002analytic,achlioptas2008algorithmic,ding2013maximum}).  Some of the most well-studied such problems involve the independent sets of a graph.  Consider an undirected graph $G$, which consists of a set of nodes $V$ and edges $E$, where each edge $e \in E$ is of the form $(v_i,v_j)$ for some $v_i,v_j \in V$.  Then the independent sets of $G$, ${\mathcal I}(G)$, are defined to be the subsets $S$ of $V$ with no internal edges; i.e. a set $S \subseteq V$ is an independent set iff for all pairs of nodes $v_i,v_j \in S$, $(v_i,v_j) \notin E$.  There are a wealth of results about the complexity and (in)approximability of counting, sampling, and optimizing independent sets under various restrictions.  We make no attempt to survey that literature here, instead focusing only on the results most relevant to our own investigations, and refer the interested reader to \cite{sly2014counting} and the references therein for a recent overview.
\subsection{Infinite-volume Gibbs measures on the Cayley tree and the uniqueness regime}
As our main results will be stated in terms of measures on the $\Delta$-regular infinite Cayley tree $T_{\infty}$, we begin by briefly reviewing several concepts needed to formally describe such measures, following the exposition given in \cite{friedli2014equilibrium}.  We assume the nodes of $T_{\infty}$ are indexed by the non-negative integers $Z^+$, and the tree is rooted at node 0.  With a slight abuse of notation, we also let $T_{\infty}$ denote the corresponding indexed set of nodes.  In the spin systems considered in this paper, each node $i \in T_{\infty}$ is assigned a spin from the set $\lbrace 0,1 \rbrace$.  Let $\Omega$ denote the collection of all $\lbrace 0,1 \rbrace$ spin assignments to the nodes of $T_{\infty}$.  For $\omega \in \Omega$ and $S \subseteq T_{\infty}$, let $\omega_S$ denote the resctriction of $\omega$ to the nodes of $S$, and $\Omega_S$ denote the collection of all $\lbrace 0,1 \rbrace$ spin assignments to the nodes of $S$.  For an event $A$, let $I(A)$ denote the corresponding indicator.  For $S \subseteq T_{\infty}$ and $\omega \in \Omega_S$, let $|\omega| \stackrel{\Delta}{=} \sum_{i \in S} \omega_{\lbrace i \rbrace}$.  Also, for a general set $S$, let $|S|$ denote the cardinality of $S$.
\\\indent For every $S \subseteq T_{\infty}$, we define a potential $\Phi_S: \Omega_S \rightarrow \bbr$, mapping the spins of $S$ to $\bbr$, where we use $\bbr$ to denote the positively extended real numbers, i.e. including $\infty$.  For all models considered, thre will exist a finite radius $R$ such that $\Phi_S = 0$ (i.e. is identically zero) for all $S$ which contain any two nodes $i,j$ at graph-theoretic distance strictly greater than $R$ in $T_{\infty}$.  Let $\mathbf{\Phi}$ denote the collection of all potentials, i.e. $\lbrace \Phi_S , S \subseteq T_{\infty} \rbrace$.  As a notational convention, let us evaluate all empty summations to zero, and all empty products to unity.  For $i,j \in T_{\infty}$, let $d(i,j)$ denote the graph-theoretic distance between $i$ and $j$ in $T_{\infty}$.  For $S \subseteq T_{\infty}$ and $i \in T_{\infty}$, let $d(i,S) \stackrel{\Delta}{=} \inf_{j \in S} d(i,j)$, and $\partial S \stackrel{\Delta}{=} \bigcup_{j \in T_{\infty}\ :\ d(j,S)\ \leq\ 2 R} \lbrace j \rbrace \setminus S$, i.e. $\partial S$ denotes the depth-$2 R$ boundary surrounding $S$.  For $d \geq 1$, let $T_d$ denote the set of nodes with graph-theoretic distance at most $d$ from $0$ in $T_{\infty}$.  For two disjoint subsets $S_1,S_2 \subseteq T_{\infty}$, and configurations $\omega^1 \in \Omega_{S_1}, \omega^2 \in \Omega_{S_2}$, let $\omega^1 \cdot \omega^2$ denote the composition spin assignment which agrees with $\omega^1$ on $S_1$ and $\omega^2$ on $S_2$.  
\\\indent For every $\Lambda \subseteq T_{\infty}$, we define the Hamiltonian ${\mathcal H}^{\mathbf{\Phi}}_{\Lambda}: \Omega_{\Lambda \bigcup \partial \Lambda} \rightarrow \bbr$ as $\sum_{S \subseteq T_{\infty}\ :\ S \bigcap \Lambda \neq \emptyset} \Phi_S(\omega_S)$.  A so-called infinite-volume Gibbs measure $\mu$ consistent with $\mathbf{\Phi}$ is a probability measure $\mu$ on $\Omega$  (associated with an appropriate probability space and filtration ${\mathcal F}$, see \cite{friedli2014equilibrium} for details), which satisifes certain consistency requirements associated with conditioning on a boundary.  In particular, for finite $S \subseteq \Lambda \subseteq T_{\infty}$, $\omega \in \Omega_S$, and $\eta \in \Omega_{\partial \Lambda}$, let 
\begin{equation}\label{lim0p5}
\pr_{\mathbf{\Phi},\Lambda}(S = \omega | \eta) \stackrel{\Delta}{=} 
\frac{\sum_{\nu \in \Omega_{\Lambda}\ :\ \nu_S\ =\ \omega} 
exp\big(- {\mathcal H}^{\mathbf{\Phi}}_{\Lambda}(\nu \cdot \eta) \big)}
{\sum_{\nu \in \Omega_{\Lambda}} exp\big(- {\mathcal H}^{\mathbf{\Phi}}_{\Lambda}(\nu \cdot \eta) \big)},
\end{equation}
whenever this ratio is well-defined.  For an event $A$ on an appropriate filtration associated with the subset $S$, we analogously define 
$\pr_{\mathbf{\Phi},\Lambda}(A | \eta) = \sum_{\omega \in A} \pr_{\mathbf{\Phi},\Lambda}(S = \omega | \eta)$.
For $S \subseteq T_{\infty}$ and $\omega \in \Omega_S$, we let $\lbrace S = \omega \rbrace$ be the event that the nodes of $S$ receive the spin-configuration dictated by $\omega$.  Then the aforementioned consistency requires that for any finite $S \subseteq \Lambda \subseteq T_{\infty}$ and $\omega \in \Omega_S$,
\begin{equation}\label{lim1}
\mu(S = \omega) = \sum_{\eta \in \partial \Lambda} \pr_{\mathbf{\Phi},\Lambda}(S = \omega | \eta) \times \mu(\partial \Lambda = \eta).
\end{equation}
Any measure $\mu$ satisfying (\ref{lim1}), as well as certain other technical conditions (the details of which we omit, instead referring the reader to \cite{friedli2014equilibrium}), is said to be an infinite-volume Gibbs measure consistent with $\mathbf{\Phi}$, and we let ${\mathcal G}(\mathbf{\Phi})$ denote the collection of all such measures.  As a notational convenience, we denote $\pr_{\mathbf{\Phi},T_d}(S = \omega | \eta)$ by $\pr_{\mathbf{\Phi}}(S = \omega | \eta)$, where $d$ is to be inferred from context (e.g. $\eta$ belonging to $\Omega_{\partial T_d}$).
\\\indent It is well-known that under minimal technical conditions ${\mathcal G}(\mathbf{\Phi})$ is a non-empty convex set, where we denote the corresponding set of extreme measures as $\hat{{\mathcal G}}(\mathbf{\Phi})$.  If $|\hat{{\mathcal G}}(\mathbf{\Phi})| = 1$, we say that $\mathbf{\Phi}$ belongs to the \emph{uniqueness regime}, i.e. admits a unique infinite-volume Gibbs measure.  Furthermore, every such extremal measure can be constructed as a so-called thermodynamic limit of appropriately conditioned finite spin systems, in the following sense.  To each $\mu \in \hat{{\mathcal G}}(\mathbf{\Phi})$, we can associate $\omega^{\mu} \in \Omega$ such that for any finite $S \subseteq T_{\infty}$ and $\omega \in \Omega_S$,
\begin{equation}\label{lim2}
\mu(S = \omega) = \lim_{d \rightarrow \infty} \pr_{\mathbf{\Phi}}(S = \omega | \omega^{\mu}_{\partial T_d}).
\end{equation}
In light of (\ref{lim2}), non-uniqueness can also be interpreted as non-vanishing dependence on distant boundary conditions.
\subsection{Hardcore model on $T_{\infty}$}
The hardcore model on $T_{\infty}$ coincides with the following collection of potentials $\mathbf{\Phi}$.    For some fixed activity $\lambda > 0$: $\Phi_{\lbrace i \rbrace}(\omega) = - \log(\lambda) I( \omega_{\lbrace i \rbrace} = 1)$ for all $i \in T_{\infty}$; $\Phi_{\lbrace i,j \rbrace}(\omega) = \infty I\big( |\omega_{\lbrace i , j \rbrace}| = 2 \big)$ for all pairs of nodes $(i,j)$ which are adjacent in $T_{\infty}$; and $\Phi_S$ is identically zero for all other $S \subseteq T_{\infty}$.  Under local conditioning, this measure puts all probability on spin assignments corresponding to independent sets, assigning an independent set $S$ probability proportional to $\lambda^{|S|}$.  When $\lambda = 1$, computing the relevant normalizing constant (i.e. partition function) is equivalent to counting the number of independent sets (a $\#P$-Complete problem in general graphs \cite{Weitz.06}); as $\lambda \rightarrow \infty$, all the probability mass gets put on the largest independent sets, and computing the partition function is analagous to finding the cardinality of the maximum independent set (an NP-Complete problem in general graphs \cite{GJ.79}).  Such models have a rich history in the physics literature.  Models on the infinite lattice were studied early-on by several authors (cf. \cite{Temp.59,Domb.60,Run.66}).  This work was extended to the three-regular infinite Cayley tree by L.K. Runnels in \cite{Run.67}, and the general $\Delta$-regular infinite Cayley tree in \cite{Colin.82}.
\subsubsection{Phase transition and non-uniqueness}
Motivated by the behavior of large particle systems, several of the original investigations of the hardcore model focused on identifying which sets of potentials (here parametrized by $\lambda$) belonged to the uniqueness regime.  In particular, for each $\Delta \geq 3$, there exists a critical actvity $\lambda_{\Delta} \stackrel{\Delta}{=} (\Delta-1)^{\Delta-1}(\Delta-2)^{-\Delta}$ such that the hardcore model on the infinite $\Delta$-regular Cayley tree admits a unique infinite-volume Gibbs measure iff $\lambda \in (0,\lambda_{\Delta}]$ (cf. \cite{Kelly.91}).  More recently, it has been shown that this same phase transition also corresponds to the point at which certain Markov chains for sampling from the independent sets of a graph of maximum degree $\Delta$ switch from mixing in polynomial time to mixing in exponential time (cf. \cite{MWW.09}).  Furthermore, it was shown in \cite{Weitz.06} that for all $\lambda \leq \lambda_{\Delta}$, the problem of computing $\sum_{S \in {\mathcal I}(G)} \lambda^{|S|}$ admits a Fully Polynomial Time Approximation Scheme (FPTAS) for all graphs of maximum degree $\Delta$.  Combined with the results of \cite{sly2014counting} (and the references therein), which show that no such FPTAS exists for $\lambda > \lambda_{\Delta}$ unless certain complexity classes collapse, this shows that the aforementioned phase transition has deep connections to computational complexity.  This phase transition also has implications for various other applications, e.g. the design of communication networks (cf. \cite{Kelly.91}).
\subsection{Higher-order M.r.f. for independent sets}
Many applications modeled by Gibbs measures defined on the independent sets of graphs involve more complicated dependencies and constraints on the independent sets themselves.  This includes several models in physics, e.g. models with next-nearest-neighbor and/or competing interactions (cf. \cite{vannimenus1981modulated}), kinetically constrained spin models (cf. \cite{KA.93}), and geometrically constrained spin models (cf. \cite{BM.01}), and we refer the interested reader to the recent survey of \cite{rozikov2013gibbs} for many more such examples.  Such measures also arise in combinatorial optimization, e.g. through the study of subfamilies of independent sets such as those in which every excluded node is adjacent to some minimal number of included nodes (cf. \cite{DPR.09,ding2013maximum}).  Closely related models have also arisen in the analysis of resource-constrained communication networks (cf. \cite{RSZM.02,LRZ.06,hellings2011tandem}).  In such networking applications, two central questions are:
\begin{itemize}
\item What is the distribution of the number of included neighbors of an excluded node? (cf. \cite{papadimitratos2008secure,khabbazian2008localized,akyildiz2002wireless})
\item Does the system exhibit long-range boundary independence? (cf. \cite{Kelly.91,RSZM.02})
\end{itemize}
Combining the above, we are led to the following question.
\begin{question}\label{Qnew1b}
When sampling from the independent sets of the infinite Cayley tree, which distributions can be attained for the number of included neighbors of any given excluded node, while staying in the uniqueness regime?
\end{question}
A good starting place is the hardcore model, for which the following result is well-known.  Let $B(n,p)$ denote a standard binomial distribution with parameters $n$ and $p$.
\begin{observation}\cite{Spitzer.75}\label{Obsnew1}
For the hardcore model on the infinite Cayley tree in the uniqueness regime, every excluded node has a number of included neighbors which follows a binomial distribution.  Exactly which binomial distributions can be acheived in this way is dictated by the phase-transition at $\lambda_{\Delta}$.  In particular, it is possible to induce a $B(\Delta,p)$ distribution on the number of included neighbors of each excluded node for any $p \in (0,(\Delta-1)^{-1}]$ in the uniqueness regime, and this characterization is tight.  
\end{observation}
\indent A natural framework for studying distributions on the independent sets of a graph with more complicated dependencies, reflected in many of the applications discussed above, is that of so-called higher-order M.r.f. (cf. \cite{TB.01}), equivalently spin systems in which the potentials $\Phi_S$ defining the Hamiltonian are non-zero for more complicated subsets of $T_{\infty}$ (i.e. not just individual nodes and edges, which corresond to first-order M.r.f.).  We note that such systems can also be analyzed as so-called factor (i.e. graphical) models with long-range interactions, and refer the reader to the excellent survey \cite{wainwright2008graphical} for an overview.
\subsubsection{Second-order M.r.f for independent sets}
In this paper, we will consider so-called second-order M.r.f. for independent sets (cf. \cite{TB.01}), in which potentials are defined on depth-1 neighborhoods, i.e. $R = 1$ (which should be assumed throughout).  Here we also assume that the potentials are translation and rotation-invariant.  In particular, for $i \in T_{\infty}$, let $N(i)$ denote the set of neighbors of $i$ in $T_{\infty}$, as well as $i$ itself; and $N_1(i) \stackrel{\Delta}{=} N(i) \setminus \lbrace i \rbrace$.  We will consider sets of potentials $\mathbf{\Phi}$ such that for some activity $\lambda > 0$ and strictly positive $(\Delta + 1)$-dimensional vector $\boldsymbol\theta = (\theta_0,\ldots,\theta_{\Delta})$, and every $i \in T_{\infty}$,
\begin{align}\label{gonefishin2}
\mathbf{\Phi}_{N(i)}(\omega) & = \begin{cases} - \log(\lambda) &\ \textrm{if}\ \omega_{\lbrace i \rbrace} = 1, |\omega_{N(i)}| = 1;\\- \log (\theta_k) &\ \textrm{if}\ \omega_{\lbrace i \rbrace} = 0, |\omega_{N(i)}| = k;\\
\infty &\ \textrm{otherwise};
\end{cases}
\end{align}
while $\mathbf{\Phi}_S$ is identically zero for all other $S \subseteq T_{\infty}$.  Thus, in addition to the hardcore constraints and activity parameter $\lambda$, we assign a different potential $- \log(\theta_k)$ for each excluded node which is adjancent to exactly $k$ included nodes.  To express the dependence on $\lambda$ and $\boldsymbol\theta$, we denote the corresponding set of potentials $\mathbf{\Phi}$ by the vector $(\lambda,\boldsymbol\theta)$.  For a given vector $\mb\theta$, let us say that $\mb\theta$ exhibits a phase transition if there exist strictly positive finite $\lambda_1 < \lambda_2$ such that $(\lambda_1,\mb\theta)$ belongs to the uniqueness regime, while $(\lambda_2,\mb\theta)$ does not belong to the uniqueness regime.  For $\mb\theta$ exhibiting a phase transition, let us define the critical activity $\lambda_{\mb\theta} \stackrel{\Delta}{=} \inf\lbrace \lambda > 0\ :\ (\lambda,\mb\theta)\ \textrm{does not belong to the uniqueness regime} \rbrace$.  We note that several of the examples mentioned earlier involving independent sets with more complicated dependency structure may be put in the framework of such second-order M.r.f.
\\\indent The hardcore model may be viewed as a special case of our model, in which $\theta_k = 1$ for all $k$.  More generally, it follows from a straightforward reduction that the case $\theta_k = \theta_0 \gamma^k$ (for some parameters $\theta_0, \gamma > 0$) also reduces to the hardcore model, albeit with activity $\lambda \theta^{-1}_0 \gamma^{\Delta}$.  Recall that a strictly positive sequence $\lbrace x_i, i = 0,\ldots,n \rbrace$ is called \emph{log-convex} if
$\frac{x_{i+1}}{x_{i}} \geq \frac{x_{i}}{x_{i-1}}$ for all $i \in \lbrace 1,\ldots,n-1 \rbrace$.  If $\mb\theta$ is log-convex, it is natural to define vectors $\underline{\mb\theta},\overline{\mb\theta}$ such that $\underline{\theta}_k = \theta_0 (\frac{\theta_1}{\theta_0})^k, \overline{\theta}_k = \theta_0 (\frac{\theta_{\Delta}}{\theta_{\Delta-1}})^k$, where log-convexity ensures that $\underline{\theta}_k \leq \theta_k \leq \overline{\theta}_k$ for all $k \in \lbrace 0,\ldots,\Delta \rbrace$, and $\frac{\underline{\theta}_{k+1}}{\underline{\theta}_k} \leq \frac{\theta_{k+1}}{\theta_k} \leq \frac{\overline{\theta}_{k+1}}{\overline{\theta}_k}$ for all $k \in \lbrace 0,\ldots,\Delta - 1 \rbrace$.  Note that $\underline{\theta}$ ($\overline{\theta}$) corresponds to the vector in which all ratios between consecutive entries are lowered (raised) to the lowest (highest) such ratio manifesting in $\mb\theta$.  By the aforementioned reduction to the hardcore model, $\lambda_{\underline{\mb\theta}} = \lambda_{\Delta} \theta_0 (\frac{\theta_0}{\theta_1})^{\Delta}$; while $\lambda_{\overline{\mb\theta}} = \lambda_{\Delta} \theta_0 (\frac{\theta_{\Delta-1}}{\theta_{\Delta}})^{\Delta}$, where we note that log-convexity ensures $\lambda_{\overline{\mb\theta}} \leq \lambda_{\underline{\mb\theta}}$.
\subsection{FKG Inequality}
A powerful tool for analyzing whether a given set of potentials belongs to the uniqueness regime are the so-called correlation inequalities, including the celebrated FKG Theorem (cf. \cite{fortuin1971correlation,pemantle2004towards}).  Roughly, the FKG Theorem proves that if a probability measure satisfies a certain supermodularity condition known as the FKG Inequality, then that measure enjoys certain monotonicity properties.  Although the FKG Theorem holds in considerable generality, we will only state the inequality and its implications as customized to the specific models considered in this paper, following the exposition given in \cite{haggstrom1997ergodicity} for a different generalization of the hardcore model.  Let us define a partial order $\tilde{\leq}$ on $\Omega$ (and appropriate restrictions) as follows.  Let $T^e_{\infty}$ denote the subset of $T_{\infty}$ consisting of the root $0$, and all nodes whose graph-theoretic distance from $0$ is even; and $T^o_{\infty} \stackrel{\Delta}{=} T_{\infty} \setminus T^e_{\infty}$.  For $S \subseteq T_{\infty}$ and $\omega^1, \omega^2 \in \Omega_S$, let us say that $\omega^1 \tilde{\leq} \omega^2$ if $\omega^1_{\lbrace i \rbrace} \leq \omega^2_{\lbrace i \rbrace}$ for all $i \in S \bigcap T^e_{\infty}$, and $\omega^1_{\lbrace i \rbrace} \geq \omega^2_{\lbrace i \rbrace}$ for all $i \in S \bigcap T^o_{\infty}$.  For $S \subseteq T_{\infty}$ and $\omega^1, \omega^2 \in \Omega_S$, let $\omega^1 \wedge \omega^2 \in \Omega_S$ denote the following spin configuration.  $\omega^1 \wedge \omega^2  _{\lbrace i \rbrace} = \min(\omega^1_{\lbrace i \rbrace}, \omega^2_{\lbrace i \rbrace})$ for $i \in S \bigcap T^e_{\infty}$; and $\omega^1 \wedge \omega^2  _{\lbrace i \rbrace} = \max(\omega^1_{\lbrace i \rbrace}, \omega^2_{\lbrace i \rbrace})$ for $i \in S \bigcap T^o_{\infty}$.  Similarly, let $\omega^1 \vee \omega^2 \in \Omega_S$ denote the following spin configuration.  $\omega^1 \vee \omega^2  _{\lbrace i \rbrace} = \max(\omega^1_{\lbrace i \rbrace}, \omega^2_{\lbrace i \rbrace})$ for $i \in S \bigcap T^e_{\infty}$; and $\omega^1 \vee \omega^2  _{\lbrace i \rbrace} = \min(\omega^1_{\lbrace i \rbrace}, \omega^2_{\lbrace i \rbrace})$ for $i \in S \bigcap T^o_{\infty}$.  Note that $\omega^1 \wedge \omega^2 \tilde{\leq} \omega^1,\omega^2 \tilde{\leq} \omega^1 \vee \omega^2$.  Let $\tilde{\Omega}$ denote the subset of $\Omega$ consistent with the hardcore constraints, i.e. $\omega \in \tilde{\Omega}$ if $|\omega_{\lbrace i,j \rbrace}| \leq 1$ whenever $d(i,j) = 1$, and define all projective notations (e.g. $\tilde{\Omega}_S$) in analogy with those for $\Omega$.  For an event $A$ belonging to an appropriate filtration, let us say that $A$ is \emph{increasing} if $\omega^1 \in A, \omega^1 \tilde{\leq} \omega^2$ implies $\omega^2 \in A$.  For example, if $S$ is a finite subset of $T^e_{\infty}$, then $\lbrace |\omega_S| = |S| \rbrace$, i.e. the event that all spins in $S$ are 1, is increasing.  In that case, the conditions of the FKG Inequality are as follows.
\begin{definition}\cite{fortuin1971correlation,haggstrom1997ergodicity}[FKG Inequality]\label{FKGdef}
The family of potentials $\boldsymbol\Phi$ satisfies the FKG Inequality on $T_{\infty}$ under partial order $\tilde{\leq}$ if for all $d \geq 0$, $\omega^1, \omega^2 \in \Omega_{T_d}$, and $\eta \in \tilde{\Omega}_{\partial T_d}$, 
\begin{equation}\label{fkg1}
\pr_{\mathbf{\Phi}}(T_d = \omega^1 \wedge \omega^2 | \eta) \times \pr_{\mathbf{\Phi}}(T_d = \omega^1 \vee \omega^2 | \eta)
 \geq 
\pr_{\mathbf{\Phi}}(T_d = \omega^1 | \eta) \times \pr_{\mathbf{\Phi}}(T_d = \omega^2 | \eta).
\end{equation}
\end{definition}
Then the celebrated FKG Theorem is as follows.
\begin{theorem}\cite{fortuin1971correlation}[FKG Theorem]\label{fkgtheorem}
If $\mathbf{\Phi}$ satisfies the FKG Inequality on $T_{\infty}$ under partial order $\tilde{\leq}$, then for any $d \geq 0$, $\eta \in \tilde{\Omega}_{\partial T_d}$, and increasing events $A,B$ belonging to the appropriate filtration, 
$$
\pr_{\mathbf{\Phi}}(A \bigcap B | \eta) \geq \pr_{\mathbf{\Phi}}(A | \eta) \times \pr_{\mathbf{\Phi}}(B | \eta).
$$
\end{theorem}
It is well-known that this monotonicity can be leveraged to reduce the question of uniqueness to the analysis of two special Gibbs measures (cf. \cite{preston1976random,haggstrom1997ergodicity}).  In particular, let $\omega^+ \in \Omega$ denote the spin configuration with $\omega^+_{\lbrace i \rbrace} = 1$ for all $i \in T^e_{\infty}$, and $\omega^+_{\lbrace i \rbrace} = 0$ for all $i \in T^o_{\infty}$; and $\omega^- \in \Omega$ denote the spin configuration with $\omega^-_{\lbrace i \rbrace} = 0$ for all $i \in T^e_{\infty}$, and $\omega^-_{\lbrace i \rbrace} = 1$ for all $i \in T^o_{\infty}$.  Then the following well-known implications of the FKG Inequality hold for the family of potentials $(\lambda, \boldsymbol\theta)$, whenever those potentials indeed satisfy the FKG Inequality.  Many of these implications hold in considerably greater generality, and we refer the interested reader to \cite{preston1976random} for a comprehensive discussion.
\begin{theorem}\cite{preston1976random,haggstrom1997ergodicity,friedli2014equilibrium}[Further implications of the FKG Inequality]\label{sandwichprop}
If the family of potentials $(\lambda,\boldsymbol\theta)$ satisfies the FKG inequality on $T_{\infty}$ under partial order $\tilde{\leq}$, then all of the following implications hold.
\begin{itemize}
\item There is a unique (up to sets of measure 0) infinite-volume Gibbs measure $\mu^+_{\lambda,\boldsymbol\theta}$ such that for every finite $S \subseteq T_{\infty}$ and $\omega \in \Omega_S$,
$$
\mu^+_{\lambda,\boldsymbol\theta}(S = \omega) = \lim_{d \rightarrow \infty} \pr_{\lambda,\boldsymbol\theta}(S = \omega | \omega^{+}_{\partial T_d}),
$$
and a unique (up to sets of measure 0)  infinite-volume Gibbs measure $\mu^-_{\lambda,\boldsymbol\theta}$ such that for every finite $S \subseteq T_{\infty}$ and $\omega \in \Omega_S$,
$$
\mu^-_{\lambda,\boldsymbol\theta}(S = \omega) = \lim_{d \rightarrow \infty} \pr_{\lambda,\boldsymbol\theta}(S = \omega | \omega^{-}_{\partial T_d});
$$
where all relevant limits appearing in the above definitions exist, and both of these measures are extremal, i.e. belong to $\hat{\mathcal G}(\lambda,\boldsymbol\theta)$.
\item For every increasing event $A$ on an appropriate filtration and $\mu \in {\mathcal G}(\lambda,\boldsymbol\theta)$, 
$$\mu^-_{\lambda,\boldsymbol\theta}(A) \leq \mu(A) \leq \mu^+_{\lambda,\boldsymbol\theta}(A).$$
\item $|{\mathcal G}(\lambda,\boldsymbol\theta)| = 1$ iff $\mu^+_{\lambda,\boldsymbol\theta}(\omega_{\lbrace 0 \rbrace} = 1) = \mu^-_{\lambda,\boldsymbol\theta}(\omega_{\lbrace 0 \rbrace} = 1)$.
Furthermore, if $|{\mathcal G}(\lambda,\boldsymbol\theta)| = 1$, then the unique such infinite-volume Gibbs measure is translation and rotation-invariant.
\end{itemize}
\end{theorem}
\indent If $(\lambda,\boldsymbol\theta)$ belongs to the uniqueness regime, we denote the corresponding unique infinite-volume Gibbs measure on $T_{\infty}$ by $\mu^*_{\lambda,\boldsymbol\theta}$.
Also, we let $p^{\lambda,\mb\theta}_+ \stackrel{\Delta}{=} \mu^*_{\lambda,\boldsymbol\theta}(\omega_{\lbrace 0 \rbrace} = 1)$, 
$p^{\lambda,\mb\theta}_k \stackrel{\Delta}{=} \mu^*_{\lambda,\boldsymbol\theta}(\omega_{\lbrace 0 \rbrace} = 0, |\omega_{N(0)}| = k)$, and $\mathbf{p}^{\lambda,\mb\theta}$ the corresponding vector.  Also, we let 
$\hat{p}^{\lambda,\mb\theta}_k \stackrel{\Delta}{=} p^{\lambda,\mb\theta}_k (1 - p^{\lambda,\mb\theta}_+)^{-1}$
denote the associated conditional distribution for the number of occupied neighbors of an unoccupied node, and 
$\hat{\mathbf{p}}^{\lambda,\mb\theta}$ the corresponding vector.
\subsection{Our contribution}
In this paper, we take a step towards answering Question\ \ref{Qnew1b}, by analyzing M.r.f. with second-order interactions on the independent sets of the infinite Cayley tree.  We prove that the associated Gibbsian specification satisfies the FKG Inequality whenever the local potentials defining the Hamiltonian satisfy a certain log-convexity condition.  Under this condition, we give necessary and sufficient conditions for the existence of a unique infinite-volume Gibbs measure in terms of an explicit system of equations, prove the existence of a phase transition, and give explicit lower and upper bounds on the associated critical activity, denoted $\underline{\lambda}_{\mb\theta}$ and $\overline{\lambda}_{\mb\theta}$ respectively, which we prove to exhibit a certain robustness.  Interestingly, we find that
$\underline{\lambda}_{\mb\theta}$ exhibits a dependence on $(\frac{\theta_{\Delta-1}}{\theta_{\Delta}})^{\Delta}$, like $\lambda_{\overline{\mb\theta}}$; while $\overline{\lambda}_{\mb\theta}$ exhibits a dependence on $(\frac{\theta_0}{\theta_1})^{\Delta}$, like $\lambda_{\underline{\mb\theta}}$.  For potentials which are small perturbations of those coinciding to the hardcore model at its critical activity $\lambda_{\Delta}$, we perform a perturbative analysis of the system of equations arising from our necessary and sufficient conditions for uniqueness, allowing us to explicitly characterize whether the resulting specification has a unique infinite-volume Gibbs measure in terms of whether these perturbations satisfy an explicit linear inequality.  Our analysis reveals an interesting non-monotonicity with regards to biasing towards excluded nodes with no included neighbors, which implies that the uniqueness regime for our model is incomparable to that suggested by $\lambda_{\overline{\mb\theta}}$ and $\lambda_{\underline{\mb\theta}}$.
\subsection{Outline of paper}
The rest of the paper proceeds as follows.  In Section\ \ref{mainsec}, we make several additional definitions and state our main results.  In Section\ \ref{fkgsec}, we prove that when $\mb\theta$ is log-convex, $(\lambda, \mb\theta)$ satisfies the FKG Inequality for all $\lambda > 0$.  In Section\ \ref{Sec1}, we rephrase the relevant probabilities and questions of interest in terms of sequences of ratios of partition functions, whose even and odd subsequences we prove to converge, and satisfy a certain system of equations.  By proving that 
the functions arising in this system of equations satisfy certain bounds and monotonicities, we derive our necessary and sufficient conditions for uniqueness.  In Section\ \ref{transie}, we prove the existence of a phase transition, and provide explicit bounds on the critical activity.  
In Section\ \ref{Sec3}, we perform a perturbative analysis of the system of equations arising from our necessary and sufficient conditions for uniqueness, allowing us to explicitly characterize whether the resulting specification has a unique infinite-volume Gibbs measure in terms of whether these perturbations satisfy an explicit linear inequality.  In Section\ \ref{Concsec}, we summarize our main results, provide a broader discussion of the potential use of higher-order M.r.f. for analyzing independent sets in graphs, and present directions for future research.
\section{Main Results}\label{mainsec}
\subsection{Potentials, probabilities, and reverse ultra log-concave measures}
Before stating our main results, we formally relate the family of potentials $(\lambda,\boldsymbol\theta)$ to the resulting occupancy probabilities $p^{\lambda,\mb\theta}_+, \hat{\mathbf{p}}^{\lambda,\mb\theta}$ in the uniqueness regime, and review the definition of reverse ultra log-concave measures.
\begin{observation}\label{relate1}
If $(\lambda,\mb\theta)$ belongs to the uniqueness regime, then the associated occupancy probabilities may be characterized as follows.  There exist $c,x \in \reals^+$ (depending only on $\lambda$ and $\mb\theta$) such that $\hat{p}^{\lambda,\mb\theta}_k  = c \theta_k {\Delta \choose k} x^k$ for $k \in \lbrace 0,\ldots,\Delta \rbrace$.  For certain natural choices of $\mb\theta$, $\hat{\mathbf{p}}^{\lambda,\mb\theta}$ corresponds exactly to a well-known family of distributions.  If $\mb\theta = \mathbf{1}$, then $\hat{\mathbf{p}}^{\lambda,\mb\theta}$ corresponds to a binomial distribution.  If $\theta _k = \frac{1}{k! {\Delta \choose k}}$ for $k \in \lbrace 0,\ldots,\Delta \rbrace$, then $\hat{\mathbf{p}}^{\lambda,\mb\theta}$ corresponds to a truncated Poisson distribution.  If $\theta _k = \frac{1}{{\Delta \choose k}}$ for $k \in \lbrace 0,\ldots,\Delta \rbrace$, then $\hat{\mathbf{p}}^{\lambda,\mb\theta}$ corresponds to a truncated geometric distribution.
\end{observation}
Recall that a strictly positive sequence $\lbrace x_i, i = 0,\ldots,n \rbrace$ is called log-convex if $\frac{x_{i+1}}{x_{i}} \geq \frac{x_{i}}{x_{i-1}}$ for all $i \in \lbrace 1,\ldots,n-1 \rbrace$, reverse ultra log-concave if the sequence $\lbrace \frac{x_i}{{\Delta \choose i}}, i = 0,\ldots,n \rbrace$ is log-convex, and convex if $x_{i+1} - x_i \geq x_i - x_{i - 1}$ for all $i \in \lbrace 1,\ldots,n-1 \rbrace$.  We say that a measure $\mu$ with support on $\lbrace 0,\ldots,\Delta \rbrace$ is reverse ultra log-concave if the sequence $\lbrace \mu(k), k = 0,\ldots,\Delta \rbrace$ is strictly positive and reverse ultra log-concave.  Then the following may be easily verified using Observation\ \ref{relate1}, and we refer the interested reader to \cite{chen2009reverse} for details and further references regarding reverse ultra log-concave measures.
\begin{observation}\label{ultra0}
If $\mb\theta$ is log-convex and $(\lambda,\mb\theta)$ belongs to the uniqueness regime, then 
$\hat{\mathbf{p}}^{\lambda,\mb\theta}$ is reverse ultra log-concave.  Furthermore, the binomial, truncated Poisson, and truncated Geometric distributions considered in Observation\ \ref{relate1} are all reverse ultra log-concave, with the corresponding choices of $\mb\theta$ log-convex.  
\end{observation}
\subsection{Main Results}
We now state our main results, and begin by formalizing the connection between log-convexity of the local potentials and the FKG Inequality.
\begin{theorem}[Log-convexity of potentials implies the FKG Inequality]\label{fkglogconvex}
If $\mb\theta$ is log-convex, then for all $\lambda > 0$, $(\lambda,\boldsymbol\theta)$ satisfies the FKG Inequality on $T_{\infty}$ under partial order $\tilde{\leq}$.
\end{theorem}
For a given vector $\mb\theta$, let 
$$
f_{\mb\theta}(x) \stackrel{\Delta}{=} \frac{ \sum_{k=0}^{\Delta - 1} \theta_{k+1} {\Delta - 1 \choose k} x^k }{ \sum_{k=0}^{\Delta - 1} \theta_k {\Delta - 1 \choose k} x^k },
$$
and
$$g_{\mb\theta}(x) \stackrel{\Delta}{=} \frac{1}{\sum_{k=0}^{\Delta - 1} \theta_k {\Delta - 1 \choose k} x^k}.$$ 
Using Theorem\ \ref{fkglogconvex}, the implications of the FKG Inequality dictated by Theorem\ \ref{sandwichprop}, and a recursive analysis of certain relevant partition functions, we prove the following necessary and sufficient conditions for uniqueness.
\begin{theorem}[Necessary and sufficient conditions for uniqueness]\label{impliesdecay1}
 If $\mb\theta$ is log-convex, the system of equations
\begin{equation}\label{limitequation1a}
x\ =\ \lambda g_{\mb\theta}(y)
f_{\mb\theta}^{\Delta - 1}(x);
\end{equation}
\begin{equation}\label{limitequation2a}
y\ =\ \lambda g_{\mb\theta}(x) f_{\mb\theta}^{\Delta - 1}(y);
\end{equation}
always has at least one non-negative solution on $\reals^+ \times \reals^+$.  $(\lambda,\mb\theta)$ belongs to the uniqueness regime iff this solution is unique.
\end{theorem}
Note that Theorem\ \ref{impliesdecay1} reduces to the well-known characterization for uniqueness in the hardcore model whenever $\theta_k = \theta_0 \gamma^k$ for some $\theta_0, \gamma > 0$, as in this special case $f_{\mb\theta} = \gamma$.
\\\indent We now prove the existence of a phase transition for every log-convex $\mb\theta$, and provide explicit bounds on the associated critical activity.
For log-convex $\mb\theta$, let 
$$\psi_{\mb\theta} \stackrel{\Delta}{=}
\max_{k = 0,\ldots,\Delta - 2} \bigg( \big( \Delta - (k+1) \big) \frac{\theta_{k+1}}{\theta_k} \bigg),$$
$$
\underline{\lambda}_{\mb\theta} \stackrel{\Delta}{=}
\Bigg( 2 \psi_{\mb\theta} \theta_0^{-1} (\frac{\theta_{\Delta}}{\theta_{\Delta-1}})^{\Delta - 2}
\bigg( 
\frac{\theta_{\Delta}}{\theta_{\Delta-1}} + 
(\Delta - 1) \big(\frac{\theta_{\Delta}}{\theta_{\Delta-1}} - \frac{\theta_1}{\theta_{0}} \big) \bigg)
\Bigg)^{-1} \geq\ \ \ \underline{\underline{\lambda}}_{\mb\theta} \stackrel{\Delta}{=} \frac{\theta_0}{2 \Delta^2} (\frac{\theta_{\Delta-1}}{\theta_{\Delta}})^{\Delta};
$$
and
$$\overline{\lambda}_{\mb\theta} \stackrel{\Delta}{=} \frac{3 \theta_0}{\Delta} \exp\big( 3 \frac{\theta_{\Delta}}{\theta_{\Delta-1}} \frac{\theta_0}{\theta_1} \big) (\frac{\theta_0}{\theta_1})^{\Delta}.$$
\begin{theorem}[Bounds on the critical activity]\label{phase1}
For any log-convex $\mb\theta$, $(\lambda,\mb\theta)$ belongs to the uniqueness regime for all $\lambda < \underline{\lambda}_{\mb\theta}$, and does not does not belong to the uniqueness regime for all $\lambda > \overline{\lambda}_{\mb\theta}$.
\end{theorem}
To gain further insight into $\underline{\lambda}_{\mb\theta}$ and $\overline{\lambda}_{\mb\theta}$, we briefly discuss the special case in which $\theta_k = \theta_0 \gamma^k$ for some $\theta_0,\gamma > 0$.  In this case, $\mb\theta = \underline{\mb\theta} = \overline{\mb\theta}$, and $\lambda_{\mb\theta} = \lambda_{\Delta} \theta_0 \gamma^{- \Delta}$, which (with $\theta_0,\gamma$ held fixed) scales like $\theta_0 \frac{e}{\Delta} \gamma^{- \Delta}$ as $\Delta \rightarrow \infty$.  Furthermore, it may be easily verified that in this case, $\underline{\lambda}_{\mb\theta} = \frac{\theta_0}{2(\Delta - 1)} \gamma^{- \Delta}$, $\overline{\lambda}_{\mb\theta} = \frac{3 e^3 \theta_0}{\Delta} \gamma^{- \Delta}$.  In particular, as $\Delta \rightarrow \infty$, both our lower and upper bound scale (up to constant factors independent of $\theta_0,\Delta,\gamma$) like $\frac{\theta_0}{\Delta} \gamma^{- \Delta}$, agreeing with the true asymptotic scaling of the critical activity.  When $\Delta$ is large but $\mb\theta$ does not have this simple factorized form, we find that
$\underline{\lambda}_{\mb\theta}$ exhibits a dependence on $(\frac{\theta_{\Delta-1}}{\theta_{\Delta}})^{\Delta}$, like $\lambda_{\overline{\mb\theta}}$; while $\overline{\lambda}_{\mb\theta}$ exhibits a dependence on $(\frac{\theta_0}{\theta_1})^{\Delta}$, like $\lambda_{\underline{\mb\theta}}$.  As log-convexity dictates that $\frac{\theta_{\Delta - 1}}{\theta_{\Delta}} \leq \frac{\theta_{0}}{\theta_{1}}$, this leads to a potentially exponentially large gap between our lower and upper bounds as one moves away from the special case in which $\theta_k = \theta_0 \gamma^k$.  Determining whether the associated phase transition is sharp, and more generally closing the gap between our lower and upper bounds, remain interesting open questions.  We also note that similar ideas (albeit connecting log-concavity of local potentials to certain negative association properties of the resulting Gibbs measure) were recently used in \cite{salez2011cavity} to prove the non-existence of a phase transition for so-called $b$-matchings on infinite graphs, and exploring further connections between our results and those of \cite{salez2011cavity} remains a direction for future research.
\\\indent We now comment briefly on several implications of Theorem\ \ref{phase1}, all of which follow from straightforward algebraic manipulations of $\underline{\lambda}_{\mb\theta}, \overline{\lambda}_{\mb\theta}$ and simple Taylor series expansions.  We first show that our lower and upper bounds, and by implication the associated critical activity, exhibit a certain form of robustness.
\begin{observation}[Robustness of bounds]\label{ooob1}
If there exists $c \in [0,\Delta]$ such that $\max(\frac{\theta_{\Delta}}{\theta_{\Delta-1}},\frac{\theta_0}{\theta_1}) \leq 1 + \frac{c}{\Delta}$, then 
$\underline{\lambda}_{\mb\theta} \geq \big(2 \exp(c) (1 + 4 c)\big)^{-1} \frac{\theta_0}{\Delta}$, and 
$\overline{\lambda}_{\mb\theta} \leq 3 \exp(12 + c) \frac{\theta_0}{\Delta}$.  
It follows that, up to constant factors independent of $\Delta$, the critical activity will scale like $\frac{\theta_0}{\Delta}$ (as $\Delta \rightarrow \infty$) for any vector $\mb\theta$ which does not deviate too much from the all ones vector.
\end{observation}
We next use our results to bound the critical activity for $\mb\theta$ corresponding to the truncated Poisson distribution.
\begin{observation}\label{ooob2}
For $\mb\theta$ such that $\theta _k = \frac{1}{k! {\Delta \choose k}}$, one has that $\underline{\lambda}_{\mb\theta} \geq \frac{1}{2} \Delta^{-1}$.  In particular, the critical activity is at least $\frac{1}{2} \Delta^{-1}$, as in the hardcore model.
\end{observation}
\indent An explicit description of when Equations\ \ref{limitequation1a} - \ref{limitequation2a} have a unique non-negative solution, and which inclusion/exclusion probabilities can be attained in this way, seems difficult in general.  However, we can develop a considerably more in-depth understanding when the relevant potentials are small perturbations of those coinciding to the hardcore model at the critical activity.  Let us fix a vector $\mathbf{c} = (\mathbf{c}_0,\ldots,\mathbf{c}_{\Delta})$.  It follows from a simple Taylor series expansion that convex perturbations of the all ones vector yield log-convex $\mb\theta$, as formalized below.
\begin{observation}\label{limlogcon}
For every vector $\mathbf{c}$, there exists $\epsilon_{\mathbf{c}} > 0$ such that for $h \in (0,\epsilon_{\mathbf{c}})$, $\mathbf{1} + \mathbf{c} h$ is log-convex iff $\mathbf{c}$ is convex.
\end{observation}
We now define a convenient notion of uniqueness for perturbations around the all ones vector, which we will use in our analysis.
\begin{definition}[Direction of (non) uniqueness]
We say that $\mathbf{c}$ is a direction of uniqueness iff there exists $\epsilon_{\mathbf{c}} > 0$ such that for all $h \in (0,\epsilon_{\mathbf{c}})$, $(\lambda_{\Delta}, \mathbf{1} + \mathbf{c} h)$ belongs to the uniqueness regime; and a direction of non-uniqueness iff there exists $\epsilon_{\mathbf{c}} > 0$ such that for all $h \in (0,\epsilon_{\mathbf{c}})$, $(\lambda_{\Delta}, \mathbf{1} + \mathbf{c} h)$ does not belong to the uniqueness regime.
\end{definition}
We now provide an explicit characterization / dichotomy theorem, classifying (almost) all convex vectors as either directions of uniqueness or directions of non-uniqueness.  For $j \in \lbrace 0,\ldots,\Delta \rbrace$, let $\Lambda_{\Delta,j} \stackrel{\Delta}{=} {\Delta \choose j} (\Delta-2)^{-j}$.  
Let 
$\mb\pi = (\pi_0,\ldots,\pi_\Delta)$  denote the vector such that for $j \in \lbrace 0,\ldots,\Delta \rbrace$,
$$\pi_j \stackrel{\Delta}{=} \Lambda_{\Delta,j} \big( (\Delta-2) + (6 - 5 \Delta) j + 2 (\Delta-1) j^2 \big).$$
Then we prove the following.
\begin{theorem}[Perturbative necessary and sufficient conditions for uniqueness]\label{whenunique}
A convex vector $\mathbf{c}$ is a direction of uniqueness if $\mb\pi \cdot \mathbf{c} < 0$, and a direction of non-uniqueness if $\mb\pi \cdot \mathbf{c} > 0$.
\end{theorem}
In particular, the hyperplane defined by $\mb\pi \cdot \mathbf{c} = 0$ represents a phase transition in the perturbation parameter space.  We note that the question of what happens at the boundary (i.e. $\mb\pi \cdot \mathbf{c} = 0$) seems to require a finer asymptotic analysis, and we leave this as an open question.  
\\\indent We now study some qualitative features of $\mb\pi$, to shed light on the set of convex directions of uniqueness, and reveal an interesting non-monotonicity of the uniqueness regime.
\begin{observation}\label{obs1}
For all $\Delta \geq 3$, $\mb\pi_0 > 0, \mb\pi_1 < 0, \mb\pi_2 < 0$, and $\mb\pi_k > 0$ for all $k \in \lbrace 3,\ldots,\Delta \rbrace$.
\end{observation}
That $\mb\pi_1 < 0, \mb\pi_2 < 0$, and $\mb\pi_k > 0$ for all $k \in \lbrace 3,\ldots,\Delta \rbrace$ makes sense at an intuitive level, since biasing towards excluded nodes which are adjacent to few (many) included nodes should tend to reduce (increase) alternation and long-range correlations.  That the cutoff occurs at exactly $k=2$  can be further justified by noting that the average number of included neighbors of an excluded node in the hardcore model, at the critical activity $\lambda_{\Delta}$, is $1 + (\Delta-1)^{-1} \in (1,2)$.  
\\\indent The counterintuitive feature of Observation\ \ref{obs1}, which seems to violate the above reasoning, is that $\mb\pi_0 > 0$, i.e. biasing towards excluded nodes with no included neighbors leads to non-uniqueness.  We note that this effect is perhaps especially surprising in light of Theorem\ \ref{phase1}, as we now explain.  Let $\mathbf{e}_0$ denote the $(\Delta + 1)$-dimensional vector whose first component is a 1, with all remaining components 0.  As it is easily verified that $\mathbf{1} + \mathbf{e}_0 h$ is log-convex for all $h \geq 0$, and $\lim_{h \rightarrow \infty} \underline{\lambda}_{\mathbf{1} + \mathbf{e}_0 h} = \infty$, we conclude that the associated uniqueness regime exhibits the following non-monotonicity.
\begin{corollary}\label{helpmehurtme}[Non-monotonicity of uniqueness regime]
For all $\Delta \geq 3$, there exist strictly positive finite constants $a_{\Delta} < b_{\Delta}$ such that $(\lambda_{\Delta},\mathbf{1} + \mathbf{e}_0 h)$ belongs to the uniqueness regime for $h = 0$ and $h \geq b_{\Delta}$, and does not belong to the uniqueness regime for $h \in (0,a_{\Delta})$.
\end{corollary}
Thus biasing a small amount towards excluded nodes with no included neighbors leads to non-uniqueness, while biasing a large amount towards excluded nodes with no included neighbors leads to uniqueness.  This non-monotonicity also sheds light on the relationship between $\lambda_{\mb\theta},\lambda_{\overline{\mb\theta}}$, and $\lambda_{\underline{\mb\theta}}$.  In particular, it would be natural to conjecture that for general log-convex $\mb\theta$, the critical activity $\lambda_{\mb\theta}$ always belongs to the interval $[\lambda_{\overline{\mb\theta}}, \lambda_{\underline{\mb\theta}}]$, i.e. that the critical activity is sandwiched between that for the model in which all ratios between consecutive entries of $\mb\theta$ are raised (lowered) to $\frac{\theta_{\Delta}}{\theta_{\Delta-1}} (\frac{\theta_1}{\theta_0})$.  However, the aforementioned non-monotonicity demonstrates that such a result cannot hold.  Indeed, if such a bound were to hold, it would imply that for all $h > 0$,
$$\lambda_{\mathbf{1} + \mathbf{e}_0 h} \geq \lambda_{\overline{\mathbf{1} + \mathbf{e}_0 h}} = (1 + h) \lambda_{\Delta} > \lambda_{\Delta},$$ 
which Corollary\ \ref{helpmehurtme} disproves.  Furthermore, although it is easily verified that for log-convex $\mb\theta$ one has $\overline{\lambda}_{\mb\theta} \geq \lambda_{\underline{\mb\theta}}$, and $\underline{\underline{\lambda}}_{\mb\theta} \leq \lambda_{\overline{\mb\theta}}$, in general $\underline{\lambda}_{\mb\theta}$ is incomparable to $\lambda_{\overline{\mb\theta}}$.
For example, considering the case that $\theta_k = 1$ for $k \in \lbrace 0,\ldots,\Delta - 1\rbrace$ and $\theta_{\Delta} = \Delta^2$, one can easily compute that $$\underline{\lambda}_{\mb\theta} = \bigg( 2 \Delta^{2 \Delta - 3} \big(\Delta^2 + (\Delta-1)(\Delta^2 - 1) \big) \bigg)^{-1} \geq \frac{1}{2 \Delta^{2 \Delta}},$$
which can be shown to be strictly greater than $\lambda_{\overline{\mb\theta}} = \frac{(\Delta-1)^{\Delta-1}}{(\Delta-2)^{\Delta}} \frac{1}{\Delta^{2 \Delta}}$ for all $\Delta \geq 8$.  However, our previous example involving $\mb{\theta} = \mathbf{1} + \mathbf{e}_0 h$ demonstrates that the opposite inequality can hold as well.  We note that several previous works in the literature on M.r.f. examine various notions of non-monotonicity (cf. \cite{LRZ.06}), and better understanding the relevant (non) monotonicities with regards to higher-order M.r.f. remains an interesting open question.
\section{Verification of FKG Inequality}\label{fkgsec}
In this section, we prove Theorem\ \ref{fkglogconvex}.  
\begin{proof}[Proof of Theorem\ \ref{fkglogconvex}]
It follows from the definition of the Hamiltonian, the symmetry of the potentials $\mathbf{\Phi} = (\lambda,\mb\theta)$, and a straightforward algebraic manipulation (cf. \cite{sheffield2005random}) that it suffices to demonstrate that for all $i \in T_{\infty}$ and $\omega^1, \omega^2 \in \Omega_{N(i)}$, 
\begin{equation}\label{showthisfkg}
\mathbf{\Phi}_{N(i)}(\omega^1 \wedge \omega^2) + \mathbf{\Phi}_{N(i)}(\omega^1 \vee \omega^2) \leq 
\mathbf{\Phi}_{N(i)}(\omega^1) + \mathbf{\Phi}_{N(i)}(\omega^2).
\end{equation}
Let us fix such a node $i$.  Let $S^{1,1} \stackrel{\Delta}{=} \lbrace j \in N_1(i) : \omega^1_{\lbrace j \rbrace} = 1 \rbrace$ (i.e. those neighbors of $i$ with spin 1 in $\omega^1$), $S^{1,0} \stackrel{\Delta}{=} N_1(i) \setminus S^{1,1}$; $S^{2,1} \stackrel{\Delta}{=} \lbrace j \in N_1(i) : \omega^2_{\lbrace j \rbrace} = 1 \rbrace$; and $S^{2,0} \stackrel{\Delta}{=} N_1(i) \setminus S^{2,1}$.   
\\\\We proceed by a case analysis.  First, suppose that for some $l \in \lbrace 1,2 \rbrace$, $\mathbf{\Phi}_{N(i)}(\omega^l) = \infty$.  In this case, (\ref{showthisfkg}) holds trivially.  Thus we subsequently suppose this situation is precluded.
\\\\We now treat the case $i \in T^e_{\infty}$, equivalently $N_1(i) \subseteq T^o_{\infty}$.  Note that from definitions, for $i \in T^e_{\infty}$, $\omega^1 \wedge \omega^2_{\lbrace i \rbrace} =  1$ iff both $\omega^1_{\lbrace i \rbrace} = 1$ and $\omega^2_{\lbrace i \rbrace} = 1$; $\omega^1 \vee \omega^2_{\lbrace i \rbrace} =  1$ iff either $\omega^1_{\lbrace i \rbrace} = 1$ or $\omega^2_{\lbrace i \rbrace} = 1$; $|\omega^1 \wedge \omega^2_{N_1(i)}| = |S^{1,1} \bigcup S^{2,1}|$; and $|\omega^1 \vee \omega^2_{N_1(i)}| = |S^{1,1} \bigcap S^{2,1}|$.
\\\\First, suppose $\omega^1_{\lbrace i \rbrace} = \omega^2_{\lbrace i \rbrace} = 1$.  In this case, $|S^{1,1}| = |S^{2,1}| = 0$, and thus $|\omega^1 \wedge \omega^2_{N_1(i)}| = |\omega^1 \vee \omega^2_{N_1(i)}| = 0$.  We conclude that both the left hand side (l.h.s.) and right hand side (r.h.s.) of (\ref{showthisfkg}) equal $- 2 \log(\lambda)$, and (\ref{showthisfkg}) holds.
\\\\Next, suppose $\omega^1_{\lbrace i \rbrace} = 1, \omega^2_{\lbrace i \rbrace} = 0$.  In this case, $\omega^1 \wedge \omega^2_{\lbrace i \rbrace} = 0, |\omega^1 \wedge \omega^2_{N_1(i)}| = |S^{2,1}|, \omega^1 \vee \omega^2_{\lbrace i \rbrace} = 1, |\omega^1 \vee \omega^2_{N_1(i)}| = 0$.  We conclude that both the l.h.s. and r.h.s. of (\ref{showthisfkg}) equal $- \log(\theta_{|S^{2,1}|}) - \log(\lambda)$, and (\ref{showthisfkg}) holds.  The case $\omega^1_{\lbrace i \rbrace} = 0, \omega^2_{\lbrace i \rbrace} = 1$ follows from a symmetric and identical argument.
\\\\Finally, suppose $\omega^1_{\lbrace i \rbrace} = 0, \omega^2_{\lbrace i \rbrace} = 0$.  In this case, $\omega^1 \wedge \omega^2_{\lbrace i \rbrace} = 0, |\omega^1 \wedge \omega^2_{N_1(i)}| = |S^{1,1} \bigcup S^{2,1}|, \omega^1 \vee \omega^2_{\lbrace i \rbrace} = 0, |\omega^1 \vee \omega^2_{N_1(i)}| = |S^{1,1} \bigcap S^{2,1}|$.  We conclude that (\ref{showthisfkg}) would hold if it were true that
\begin{equation}\label{showthisfkg2}
\log(\theta_{|S^{1,1} \bigcup S^{2,1}|}) + \log(\theta_{|S^{1,1} \bigcap S^{2,1}|}) \geq 
\log(\theta_{|S^{1,1}|}) + \log(\theta_{|S^{2,1}|}).
\end{equation}
However, (\ref{showthisfkg2}) follows immediately from the log-convexity of $\mb\theta$, and the well-known connection between convexity and supermodularity (cf. \cite{topkis1998supermodularity}), completing the proof for the case $i \in T^e_{\infty}$.  As the case $i \in T^o_{\infty}$ follows from a nearly identical argument with the role of $\omega^1 \vee \omega^2$ and $\omega^1 \wedge \omega^2$ reversed, we omit the details.  Combining the above completes the proof.
\end{proof}
\section{Probabilities, partition functions, and proof of Theorem\ \ref{impliesdecay1}}\label{Sec1}
In this section, we first rephrase the relevant probabilities and questions of interest in terms of sequences of ratios of partition functions, whose even and odd subsequences we prove to converge, and which are amenable to a recursive analysis.  We then prove that the functions arising in the relevant recursions satisfy certain bounds and monotonicities, which we exploit to prove Theorem\ \ref{impliesdecay1}.  Without loss of generality, let us assign the neighbors of $0$ in $T_{\infty}$ indices $1,\ldots,\Delta$ in an arbitrary but fixed manner.  For $d \geq 2$, let $T^1_d$ denote the subtree of $T_d$ rooted at node $1$, excluding node $1$ itself, i.e. the collection of nodes $j \in T_d \setminus \lbrace 1 \rbrace$ such that every undirected path in $T_d$ from $j$ to $0$ contains node $1$.  For $d \geq 2$ and $i,j \in \lbrace 0,1 \rbrace$ such that $i + j \leq 1$, let $\eta^{i,j,d} \in \Omega_{\partial T^1_d}$ denote that boundary condition such that $\eta^{i,j,d}_{\lbrace 0 \rbrace} = i, \eta^{i,j,d}_{\lbrace 1 \rbrace} = j$, $\eta^{i,j,d}_{\lbrace k \rbrace} = 0$ for all $k \in \partial T^1_d$ such that $d(k,0) = d + 1$, and $\eta^{i,j,d}_{\lbrace k \rbrace} = 1$ for all $k \in \partial T^1_d$ such that $d(k,0) = d + 2$.  Similarly, for $d \geq 2$, let $Z_{\lambda,\mb\theta,d}(i,j) \stackrel{\Delta}{=} \sum_{\nu \in \Omega_{T^1_d}} exp\big(- {\mathcal H}^{\lambda,\mb\theta}_{T^1_d}(\nu \cdot \eta^{i,j,d}) \big)$.  For the special case $d = 1$, as $T^1_d = \emptyset$, we define $Z_{\lambda,\mb\theta,1}(0,0) = \theta_0 \theta^{\Delta - 1}_{\Delta - 1}, Z_{\lambda,\mb\theta,1}(1,0) = \theta_1 \theta^{\Delta - 1}_{\Delta - 1}, Z_{\lambda,\mb\theta,1}(0,1) = \lambda \theta^{\Delta - 1}_{\Delta}$.  For $d \geq 1$ and $i \in \lbrace 0,1 \rbrace$, let $Z_{\lambda,\mb\theta,d}(i) \stackrel{\Delta}{=} \frac{ Z_{\lambda,\mb\theta,d}(i,0) }{ Z_{\lambda,\mb\theta,d}(0,1)}$,
and $\zeta_{\lambda,\mb\theta,d} \stackrel{\Delta}{=} Z^{-1}_{\lambda,\mb\theta,d}(0)$, where we also define $\zeta_{\lambda,\mb\theta,0} \stackrel{\Delta}{=} 0$.  When there is no ambiguity, we will supress the notation on $(\lambda,\mb\theta)$, simply writing e.g. $Z_d(i,j), Z_d(i),\zeta_d,f,g$.  For $d \geq 1$, let $\eta^{-,+,d} \in \Omega_{\partial T_d}$ denote that boundary condition such that $\eta^{-,+,d}_{\lbrace k \rbrace} = 0$ for all $k \in \partial T_d$ such that $d(k,0) = d + 1$, and $\eta^{-,+,d}_{\lbrace k \rbrace} = 1$ for all $k \in \partial T_d$ such that $d(k,0) = d + 2$.  Note that for even $d$, $\eta^{-,+,d} = \omega^+_{\partial T_d}$; while for odd $d$, $\eta^{-,+,d} = \omega^-_{\partial T_d}$.  Then for $d \geq 2$ and $k \in \lbrace 0,\ldots,\Delta \rbrace$,
\begin{eqnarray}
\pr_{\lambda,\bf\theta}\big(\omega_0 = 0 , |\omega_{N(0)}| = k \big| \eta^{-,+,d}\big) &=& 
\frac
{
{\Delta \choose k} \theta_k Z_d^k(0,1) Z_d^{\Delta - k}(0,0)
}
{
\lambda Z_d^{\Delta}(1,0) + 
\sum_{i = 0}^{\Delta} {\Delta \choose i} \theta_i Z_d^i(0,1) Z_d^{\Delta - i}(0,0)
}
\nonumber
\\&=&
\frac
{
{\Delta \choose k} \theta_k \zeta_d^k
}
{
\lambda \big(Z_d(1) \zeta_d\big)^{\Delta}
+ 
\sum_{i=0}^{\Delta} {\Delta \choose i} \theta_i \zeta_d^i,
}
\stackrel{\Delta}{=} p^{\lambda,\mb\theta,d}_k. \label{rree2}
\end{eqnarray}
We also let $\mathbf{p}^{\lambda,\mb\theta,d}$ denote the associated vector, and $p^{\lambda,\mb\theta,d}_+ \stackrel{\Delta}{=} 1 - \mathbf{p}^{\lambda,\mb\theta,d} \cdot \mathbf{1}$.  
\\\indent We now derive several recursions for $Z_d(i)$ and $\zeta_d$, to aid in our analysis.  
\begin{lemma}\label{lemma3}
For all $d \geq 1$, 
\begin{equation}\label{lemma3a}
Z_d(1) = \zeta^{-1}_d f(\zeta_{d-1});
\end{equation}
and for all $d \geq 2$,
\begin{equation}\label{lemma3b}
\zeta_d = \lambda g(\zeta_{d-1}) f^{\Delta - 1}(\zeta_{d-2}).
\end{equation}
\end{lemma}
\begin{proof}
We first treat the cases $d = 1,2$.  That (\ref{lemma3a}) holds for $d = 1$ follows from definitions and the fact that $f(0) = \frac{\theta_1}{\theta_0}$.  For $d = 2$, a  straightforward calculation demonstrates that
$$Z_2(0,0) = (\theta^{\Delta - 1}_{\Delta - 1} \theta_0)^{\Delta - 1} \sum_{k = 0}^{\Delta - 1} \theta_k {\Delta - 1 \choose k} \zeta^k_1\ \ \ ,\ \ \ 
Z_2(1,0) = (\theta^{\Delta - 1}_{\Delta - 1} \theta_0)^{\Delta - 1} \sum_{k = 0}^{\Delta - 1} \theta_{k+1} {\Delta - 1 \choose k} \zeta^k_1,$$
$$Z_2(0,1) = \lambda \theta^{\Delta - 1}_1 \theta^{(\Delta-1)^2}_{\Delta-1}\ \ \ ,\ \ \ 
Z_2(1) = \lambda^{-1} (\frac{\theta_0}{\theta_1})^{\Delta - 1}  \sum_{k = 0}^{\Delta - 1} \theta_{k+1} {\Delta - 1 \choose k} \zeta^k_1,$$
$$Z_2(0) = \lambda^{-1} (\frac{\theta_0}{\theta_1})^{\Delta - 1}  \sum_{k = 0}^{\Delta - 1} \theta_{k} {\Delta - 1 \choose k} \zeta^k_1\ \ \ ,\ \ \ 
\zeta_2 = \lambda f^{\Delta - 1}(0) g(\zeta_1);$$
from which (\ref{lemma3a}) and (\ref{lemma3b}) follow.  
\\\indent For $d \geq 3$ and $i \in \lbrace 0,1 \rbrace$,
$$Z_d(i,0) = \sum_{k = 0}^{\Delta - 1} {\Delta - 1 \choose k} \theta_{k + i} 
Z^k_{d-1}(0,1) Z^{\Delta - 1 - k}_{d-1}(0,0),$$
and
$$Z_d(0,1) = \lambda Z^{\Delta - 1}_{d-1}(1,0).$$
Thus
\begin{eqnarray*}
Z_d(0) &=& 
\frac
{ \sum_{k = 0}^{\Delta - 1} {\Delta - 1 \choose k} \theta_{k} 
Z^k_{d-1}(0,1) Z^{\Delta - 1 - k}_{d-1}(0,0)
}
{\lambda Z^{\Delta - 1}_{d-1}(1,0)}
\\&=& \lambda^{-1} (\frac{Z_{d-1}(0)}{Z_{d-1}(1)})^{\Delta - 1}
\sum_{k=0}^{\Delta - 1} {\Delta - 1 \choose k} \theta_k \zeta^k_{d-1}.
\end{eqnarray*}
Similarly,
$$
Z_d(1) = \lambda^{-1} (\frac{Z_{d-1}(0)}{Z_{d-1}(1)})^{\Delta - 1}
\sum_{k=0}^{\Delta - 1} {\Delta - 1 \choose k} \theta_{k+1} \zeta^k_{d-1}.
$$
Combining with the definition of $f$ and $g$ completes the proof.
\end{proof}
We next establish some useful properties of $f$ and $g$.
\begin{lemma}\label{monlemma}
If $\mb\theta$ is log-convex, then for all $x \geq 0$ : $\partial_x f(x) \geq 0$, $\partial_x g(x) \leq 0$, $0 \leq g(x) \leq \theta_0^{-1}$, and 
$\frac{\theta_{1}}{\theta_{0}} \leq f(x) \leq \frac{\theta_{\Delta}}{\theta_{\Delta-1}}$.
\end{lemma}
\begin{proof}
We first prove that $\partial_x f(x) \geq 0$ for all $x \geq 0$.  It follows from a straightforward calculation that
\begin{equation}\label{fgoodrep1}
\partial_x f(x)
=
(\Delta - 1) 
\frac{ 
\sum_{i=0}^{\Delta - 1} \sum_{j = 0}^{\Delta - 2} {\Delta - 1 \choose i} {\Delta - 2 \choose j} x^{i + j} ( \theta_i \theta_{j + 2} - \theta_{i + 1} \theta_{j + 1})
}
{
\big(\sum_{k=0}^{\Delta - 1} \theta_k {\Delta - 1 \choose k} x^k \big)^2
},
\end{equation}
and $\sum_{i=0}^{\Delta - 1} \sum_{j = 0}^{\Delta - 2} {\Delta - 1 \choose i} {\Delta - 2 \choose j} x^{i + j} ( \theta_i \theta_{j + 2} - \theta_{i + 1} \theta_{j + 1})$ equals
\begin{equation}\label{showme1}
\sum_{k = 0}^{2 \Delta - 3} x^k \sum_{i = \max\big(0, k - (\Delta - 2) \big) }^{\min(\Delta - 1, k)} {\Delta - 1 \choose i} {\Delta - 2 \choose k - i} ( \theta_i \theta_{k - i + 2} - \theta_{i + 1} \theta_{k - i + 1}).
\end{equation}
We now demonstrate that 
\begin{equation}\label{showme2}
\sum_{i = \max\big(0, k - (\Delta - 2) \big) }^{\min(\Delta - 1, k)} {\Delta - 1 \choose i} {\Delta - 2 \choose k - i} ( \theta_i \theta_{k - i + 2} - \theta_{i + 1} \theta_{k - i + 1})
\end{equation}
 is non-negative for all $k \in [0, 2 \Delta - 3]$, completing the proof.  We proceed by ``pairing up" certain terms appearing in (\ref{showme2}), by using the fact that for all $k \in [0, 2 \Delta - 3]$ and $i \in \big[\max\big(0, k - (\Delta - 2)\big) , \min(\Delta-1, k) \big]$,
\begin{equation}\label{useme3}
\theta_i \theta_{k - i + 2} - \theta_{i + 1} \theta_{k - i + 1} = - \big( \theta_{(k - i + 1)} \theta_{k - (k - i + 1) + 2} - \theta_{(k - i + 1) + 1} \theta_{k - (k - i + 1)+ 1} \big).
\end{equation}
Let us say that a function $f$, with domain and range containing the finite set $S \subseteq Z$, is a paired bijection on $S$ if the restriction of $f$ to domain $S$ is a bijection (i.e. every element of $S$ is mapped to some element of $S$, and every element of $S$ is mapped to by some element of $S$), and $f$ does not map any element of $S$ to itself.  We now demonstrate that for all $k \in [0, 2 \Delta - 3]$, the mapping 
$f(i) = k - i + 1$ is a paired bijection on $S = \big[\max\big(0, k - (\Delta - 2)\big) , \min(\Delta-1, k) \big] \setminus \lbrace 0, \frac{k+1}{2} \rbrace$.  We first show $i \in S$ implies $f(i) \in S$.  Indeed, $i \geq 0, i \neq 0$ implies $i \geq 1$ and thus $f(i) \leq k$; $i \geq k - (\Delta - 2)$ implies $f(i) \leq \Delta - 1$; $i \leq \Delta - 1$ implies $f(i) \geq k - (\Delta - 2)$; $i \leq k$ implies $f(i) \geq 1$.  Furthermore, as $\frac{k+1}{2}$ is the unique (possibly non-integer) solution to $f(x) = x$ (i.e. fixed point), $i \neq \frac{k+1}{2}$ implies $f(i) \neq \frac{k+1}{2}$.  Combining the above completes the demonstration that $i \in S$ implies $f(i) \in S$.  The proof that $f$ is a paired bijection on $S$ then follows from the fact that $f$ is strictly decreasing and invertible, with unique fixed point $\frac{k+1}{2}$.  Furthermore, since $f$ is strictly decreasing and $\lfloor \frac{k}{2} \rfloor < \frac{k+1}{2} \leq \lfloor \frac{k}{2} \rfloor + 1$, it is also true that 
for $i \in S$, $i \leq \lfloor \frac{k}{2} \rfloor$ iff $f(i) \geq \lfloor \frac{k}{2} \rfloor + 1$.  Combining the above with (\ref{useme3}) and the fact that $\lfloor \frac{k}{2} \rfloor \leq \min(\Delta - 1, k)$ for all $k \in [0, 2 \Delta - 3]$, we conclude that (\ref{showme2}) equals
\begin{eqnarray}
\ &\ &\ \sum_{i = \max\big(1, k - (\Delta - 2) \big) }^{\lfloor \frac{k}{2} \rfloor} \Bigg({\Delta - 1 \choose i} {\Delta - 2 \choose k - i} - {\Delta - 1 \choose k - i + 1} {\Delta - 2 \choose i - 1} \Bigg) \big( \theta_i \theta_{k - i + 2} - \theta_{i + 1} \theta_{k - i + 1} \big) \label{fderiv1}
\\&\ &\ \ \ +\ \ \ I(k \leq \Delta - 2) {\Delta - 2 \choose k} (\theta_0 \theta_{k + 2} - \theta_1 \theta_{k + 1}) \label{fderiv2}
\\&\ &\ \ \ +\ \ \ I(\frac{k+1}{2} \in Z^+) {\Delta - 1 \choose \frac{k+1}{2}} {\Delta - 2 \choose \frac{k - 1}{2}} \big( \theta_{\frac{k+1}{2}} \theta_{\frac{k+1}{2} + 1} - \theta_{\frac{k+1}{2} + 1} \theta_{\frac{k+1}{2}} \big). \label{fderiv3}
\end{eqnarray}
We now verify that (\ref{fderiv1}) - (\ref{fderiv3}) are non-negative, and begin with (\ref{fderiv1}).  It follows from a straightforward calculation that ${\Delta - 1 \choose i} {\Delta - 2 \choose k - i} - {\Delta - 1 \choose k - i + 1} {\Delta - 2 \choose i - 1}$ will be the same sign as 
$\frac{k+1}{i} - 2$, and thus non-negative for $i \leq \lfloor \frac{k}{2} \rfloor$.  Also, the log-convexity of $\mb\theta$ implies that $\theta_i \theta_{k - i + 2} - \theta_{i + 1} \theta_{k - i + 1}$ will be non-negative if $k - i + 1 \geq i$, which holds for $i \leq \lfloor \frac{k}{2} \rfloor$.  Combining the above demonstrates the non-negativity of (\ref{fderiv1}).  The log-convexity of $\mb\theta$ similarly implies the non-negativity of $\theta_0 \theta_{k + 2} - \theta_1 \theta_{k + 1}$, and thus also of (\ref{fderiv2}).  As (\ref{fderiv3}) is identically zero, combining the above completes the proof.  That $\frac{\theta_{1}}{\theta_{0}} \leq f(x) \leq \frac{\theta_{\Delta}}{\theta_{\Delta-1}}$ then follows by letting $x \downarrow 0$ and $x \uparrow \infty$.  Noting that the associated monotonicity and bounds for $g$ are straightforward completes the proof of the lemma.
\end{proof}
We now combine Lemmas\ \ref{lemma3} and\ \ref{monlemma} to prove that the even and odd subsequences of $\lbrace \zeta_d, d \geq 1 \rbrace$ are monotone and thus converge, where we will later prove that $(\lambda,\mb{\theta})$ belongs to the uniqueness regime iff these limits coincide.  We note that although the monotonicity of certain related sequences follows directly from the FKG Theorem and its implied monotonicities, here our analysis fundamentally involves $Z_d(0,0)$, in which nodes at both even and odd parity have their spins set to 0, which seems to preclude such a direct approach.  Instead, we proceed by induction, using the properties of $f$ and $g$ demonstrated in Lemma\ \ref{monlemma}.
\begin{lemma}\label{theorem2}
If $\mb\theta$ is log-convex, then $\lbrace \zeta_{2 d}, d \geq 0 \rbrace$ is monotone increasing, and $\lbrace \zeta_{2 d + 1}, d \geq 0 \rbrace$ is monotone decreasing.
\end{lemma}
\begin{proof}
We proceed by induction simultaneously on both sequences.  The base case entails demonstrating that 
$\zeta_3 \leq \zeta_1$, and $\zeta_2 \geq \zeta_0$.  It follows from definitions that 
$\zeta_1 = \lambda \theta^{-1}_0 (\frac{\theta_{\Delta}}{\theta_{\Delta-1}})^{\Delta - 1}$.  Combining with the bounds for $f$ and $g$ of Lemma\ \ref{monlemma}, and the fact that 
$\zeta_3$ satisfies (\ref{lemma3b}), demonstrates that $\zeta_3 \leq \zeta_1$.  That $\zeta_2 \geq \zeta_0$ follows from non-negativity, completing the proof of the base case.
\\\indent Now, suppose that $\lbrace \zeta_{2 k}, k = 0,\ldots,d - 1 \rbrace$ is monotone increasing, and $\lbrace \zeta_{2 k + 1}, k = 0,\ldots, d - 1 \rbrace$ is monotone decreasing for some $d \geq 2$.  Then it follows from Lemma\ \ref{lemma3}, the monotonicity of $f$ and $g$ guaranteed by Lemma\ \ref{monlemma}, and the induction hypothesis that 
\begin{eqnarray*}
\zeta_{2 d} &=&
\lambda g(\zeta_{2(d - 1) + 1}) f^{\Delta-1}(\zeta_{2 (d - 1)})
\\&\geq& \lambda g(\zeta_{2(d-2) + 1}) f^{\Delta-1}(\zeta_{2(d - 2)})\ \ \ =\ \ \ \zeta_{2(d-1)}.
\end{eqnarray*}
Similarly, the above further implies that
\begin{eqnarray*}
\zeta_{2 d + 1} &=& \lambda g(\zeta_{2d}) f^{\Delta-1}(\zeta_{2 (d - 1) + 1})
\\&\leq& \lambda g(\zeta_{2(d-1)}) f^{\Delta-1}(\zeta_{2(d - 2) + 1})\ \ \ =\ \ \ \zeta_{2(d-1) + 1}.
\end{eqnarray*}
Combining the above completes the proof.
\end{proof}
It follows from Lemma\ \ref{theorem2} that $\overline{\zeta}_{\lambda,\mb\theta,\infty} \stackrel{\Delta}{=} \lim_{d \rightarrow \infty}\zeta_{\lambda,\mb\theta,2 d + 1}$ and $\underline{\zeta}_{\lambda,\mb\theta,\infty} \stackrel{\Delta}{=} \lim_{d \rightarrow \infty} \zeta_{\lambda,\mb\theta,2 d}$ both exist.  Furthermore, the continuity of $f$ and $g$ on $\reals^+$, combined with (\ref{lemma3b}) and Lemma\ \ref{theorem2}, implies the following. 
\begin{observation}\label{usefor1}
$0 < \underline{\zeta}_{\infty},\overline{\zeta}_{\infty} < \infty$, and $(\underline{\zeta}_{\infty} , \overline{\zeta}_{\infty})$ is a solution to the system of equations (\ref{limitequation1a}) - (\ref{limitequation2a}).
\end{observation}
With Observation\ \ref{usefor1} in hand, we now complete the proof of Theorem\ \ref{impliesdecay1}, as well as Observation\ \ref{relate1}.  Let $L_{\mb\theta}(z) \stackrel{\Delta}{=} \sum_{i=0}^{\Delta} \theta_i {\Delta \choose i} z^i$.
\begin{proof}[Proof of Theorem\ \ref{impliesdecay1}]
We first prove that the system of equations  (\ref{limitequation1a}) - (\ref{limitequation2a}) always has at least one solution $(x^*,y^*)$ on $\reals^+ \times \reals^+$ for which $x^* = y^*$.  Let $\eta(x) \stackrel{\Delta}{=} x - \lambda g(x) f^{\Delta-1}(x)$.  Note that $\eta(0) = - \lambda \frac{\theta^{\Delta-1}_1}{\theta^{\Delta}_0} < 0$.  It follows from Lemma\ \ref{monlemma} that $\eta\bigg( \lambda \theta^{-1}_0 \big(\frac{\theta_{\Delta}}{\theta_{\Delta-1}}\big)^{\Delta-1} \bigg) \geq 0$.  As $\eta$ is continuous on $[0,\infty)$, we conclude that there exists $z^* \in \reals^+$ such that $\eta(z^*) = 0$, which implies that $(z^*,z^*)$ is a solution to the system of equations.
\\\indent We now prove that if the system of equations  (\ref{limitequation1a}) - (\ref{limitequation2a}) has a unique solution on $\reals^+ \times \reals^+$, then $(\lambda,\mb\theta)$ belongs to the uniqueness regime.  Suppose (\ref{limitequation1a}) - (\ref{limitequation2a}) has a unique solution $(x^*,y^*)$ on $\reals^+ \times \reals^+$.  Then it must be that any non-negative solution $(x,y)$ to the system of equations satisfies $x = x^* = y^* = y$.  By Observation\ \ref{usefor1}, $(\underline{\zeta}_{\infty},\overline{\zeta}_{\infty})$ is such a solution.  Thus $\underline{\zeta}_{\infty} = \overline{\zeta}_{\infty}$, in which case it follows from Theorems\  \ref{fkglogconvex} and \ref{sandwichprop}, (\ref{rree2}), and Lemma\ \ref{lemma3} that $(\lambda,\mb\theta)$ belongs to the uniqueness regime.
\\\indent We now prove that if the system of equations  (\ref{limitequation1a}) - (\ref{limitequation2a}) does not have a unique solution on $\reals^+ \times \reals^+$, then $(\lambda,\mb\theta)$ does not belong to the uniqueness regime.  Indeed, suppose that the system of equations does not have a unique solution on $\reals^+ \times \reals^+$.  Let $S$ denote the set of all 2-vectors $(x,y)$ such that $0 \leq x \leq y < \infty$, and $(x,y)$ is a solution to the system of equations.  Let $\overline{y} \stackrel{\Delta}{=} \sup_{\mathbf{z} \in S} z_2$, i.e. the largest number appearing in any solution pair.  We first show that $\overline{y}$ is itself part of some solution pair (i.e. it is not just approached).  If $|S| < \infty$, this is immediate.  If not, consider any sequence of solution vectors $\lbrace \mathbf{z}^i , i \geq 1 \rbrace$ such that $\lim_{i \rightarrow \infty} z^i_2 = \overline{y}$.  Since $\lbrace z^i_1, i \geq 1 \rbrace$ is uniformly bounded by Lemma\ \ref{monlemma}, the Bolzano-Weirerstrass Theorem implies that  $\lbrace \mathbf{z}^i , i \geq 1 \rbrace$ will itself have a convergent subsequence $\lbrace \mathbf{z}^{i_k}, k \geq 1 \rbrace$, and let us denote $\lim_{k \rightarrow \infty} z^{i_k}_1$ by $\overline{x}$.  That $(\overline{x},\overline{y})$ satisfies the system of equations then follows from the continuity of $f$ and $g$.  Similarly, let $\underline{x} \stackrel{\Delta}{=} \inf_{\mathbf{z} \in S} z_1$, i.e. the smallest number appearing in any solution pair, and $\underline{y}$ the other number appearing in the corresponding solution pair (whose existence is guaranteed by the same argument used above).  Note that $\underline{x} < \overline{y}$.
\\\indent We now prove (by induction) that in this case, $\lbrace \zeta_d, d \geq 0 \rbrace $ has a non-vanishing parity-dependence, with even values lying below $\underline{x}$, and odd values lying above $\overline{y}$.  We begin with the base cases $d=0,1$.  The $d = 0$ case follows from non-negativity.  The $d = 1$ case follows from the fact that $\overline{y}$ satisfies (\ref{limitequation2a}), combined with Lemma\ \ref{monlemma} and the definition of $\zeta_1$.  Now, proceeding by induction, suppose that for some $d \geq 1$ and all $k \in \lbrace 0,\ldots, d - 1 \rbrace$, $\zeta_{2 k} \leq \underline{x}$, and $\zeta_{2 k + 1} \geq \overline{y}$.  Then it follows from Lemma\ \ref{lemma3}, and the monotonicity of $f$ and $g$ implied by Lemma\ \ref{monlemma}, that
\begin{eqnarray*}
\zeta_{2d} &=& \lambda g(\zeta_{2d - 1}) f^{\Delta - 1}(\zeta_{2d - 2})
\\&\leq& \lambda g(\overline{y}) f^{\Delta - 1}(\underline{x}) 
\\&\leq& \lambda g(\underline{y}) f^{\Delta - 1}(\underline{x} )\ \ \ =\ \ \ \underline{x}\ \ ,\ \ \textrm{since}\ \overline{y} \geq \underline{y},
\end{eqnarray*}
and
\begin{eqnarray*}
\zeta_{2d + 1} &=& \lambda g(\zeta_{2d}) f^{\Delta - 1}(\zeta_{2d - 1})
\\&\geq& \lambda g( \underline{x} ) f^{\Delta - 1}( \overline{y} ) 
\\&\geq& \lambda g( \overline{x} ) f^{\Delta - 1}( \overline{y} )\ \ \ =\ \ \ \overline{y}\ \ ,\ \ \textrm{since}\ \underline{x} \leq \overline{x},
\end{eqnarray*}
completing the proof.
\\\indent Finally, we prove that the aforementioned parity dependence of $\lbrace \zeta_d, d \geq 0 \rbrace$ implies a non-vanishing parity dependence on the probability that the root is included when conditioning on the appropriate extremal boundary conditions, implying non-uniqueness.  It follows from the parity dependence of $\lbrace \zeta_d, d \geq 0 \rbrace$, (\ref{rree2}), and Lemma\ \ref{monlemma} that for all $d \geq 1$, 
$$
\pr_{\lambda,\bf\theta}\big(\omega_0 = 1 | \eta^{-,+,2 d}\big) = 
\big( 1 + \frac{L(\zeta_{2d})}{\lambda f^{\Delta}(\zeta_{2d - 1})} \big)^{-1}
\geq  \big( 1 + \frac{L(\underline{x})}{\lambda f^{\Delta}(\overline{y})} \big)^{-1};$$
$$
\pr_{\lambda,\bf\theta}\big(\omega_0 = 1 | \eta^{-,+,2 d + 1}\big) =
\big( 1 + \frac{L(\zeta_{2 d + 1})}{\lambda f^{\Delta}(\zeta_{2 d})} \big)^{-1}
\leq  \big( 1 + \frac{L(\overline{y})}{\lambda f^{\Delta}(\underline{x})} \big)^{-1} 
< \big( 1 + \frac{L(\underline{x})}{\lambda f^{\Delta}(\overline{y})} \big)^{-1}.
$$
Combining with Theorems\ \ref{fkglogconvex} and\ \ref{sandwichprop}, along with the fact that
$\eta^{-,+,d}$ equals $\omega^+_{\partial T_d}$ for even $d$, and equals $\omega^-_{\partial T_d}$ for odd $d$, completes the proof.  As it follows that $(\lambda,\mb\theta)$ belongs to the uniqueness regime iff $\lim_{d \rightarrow \infty} \zeta_d$ exists, combining with (\ref{rree2}) also completes the proof of Observation\ \ref{relate1}.
\end{proof}
\section{Existence of phase transition and proof of Theorem\ \ref{phase1}}\label{transie}
In this section, we show the existence of a phase transition for log-convex $\mb\theta$, and provide explicit bounds on the critical activity, completing the proof of Theorem\ \ref{phase1}.  Recall that
$$\psi_{\mb\theta} = 
\max_{k = 0,\ldots,\Delta - 2} \bigg( \big( \Delta - (k+1) \big) \frac{\theta_{k+1}}{\theta_k} \bigg)\ \ \ ,\ \ \ 
\overline{\lambda}_{\mb\theta} = \frac{3 \theta_0}{\Delta} \exp\big( 3 \frac{\theta_{\Delta}}{\theta_{\Delta-1}} \frac{\theta_0}{\theta_1} \big) (\frac{\theta_0}{\theta_1})^{\Delta},$$
$$\underline{\lambda}_{\mb\theta} =
\Bigg( 2 \psi_{\mb\theta} \theta_0^{-1} (\frac{\theta_{\Delta}}{\theta_{\Delta-1}})^{\Delta - 2}
\bigg( 
\frac{\theta_{\Delta}}{\theta_{\Delta-1}} + 
(\Delta - 1) \big(\frac{\theta_{\Delta}}{\theta_{\Delta-1}} - \frac{\theta_1}{\theta_{0}} \big) \bigg)
\Bigg)^{-1}.
$$
\begin{proof}[Proof of Theorem\ \ref{phase1}]
We first show that for any log-convex $\mb\theta$, $(\lambda,\mb\theta)$ belongs to the uniqueness regime for all activities $\lambda < \underline{\lambda}_{\mb\theta}$.  We proceed by 
by proving that for $\lambda < \underline{\lambda}_{\mb\theta}$, the update rule for 
$|\zeta_d - \zeta_{d-1}|$ implied by (\ref{lemma3b}) is a contraction.  We first demonstrate that $f,g$ are Lipschitz, and explicitly bound the relevant Lipschitz constants.  Note that $f,g$ are differentiable on ${\mathcal R}^+$.  We begin by bounding $|\partial_x g(x)|$.  For all $x \geq 0$, 
$$
|\partial_x g(x) |
=  
\frac{ \sum_{k=0}^{\Delta - 2} (k+1) \theta_{k+1} {\Delta - 1 \choose k + 1} x^{k} }
{\bigg( \sum_{k=0}^{\Delta - 1} \theta_k {\Delta - 1 \choose k} x^k \bigg)^{2}}.
$$
Combining with the fact that for all $x \geq 0$, 
$$\sum_{k=0}^{\Delta - 1} \theta_k {\Delta - 1 \choose k} x^k \geq \max\bigg(\theta_0,
\sum_{k=0}^{\Delta - 2} \theta_k {\Delta - 1 \choose k} x^k\bigg),$$
it follows that
\begin{eqnarray}
|\partial_x g(x)|
&\leq&
\theta_0^{-1}
\frac{ \sum_{k=0}^{\Delta - 2} (k+1) \theta_{k+1} {\Delta - 1 \choose k + 1} x^k }
{\sum_{k=0}^{\Delta - 2} \theta_k {\Delta - 1 \choose k} x^k} \nonumber
\\&=&
\theta_0^{-1}
\frac{\sum_{k=0}^{\Delta - 2} \frac{(k+1) \theta_{k+1} {\Delta - 1 \choose k + 1}}{\theta_k {\Delta - 1 \choose k}}\big(\theta_k {\Delta - 1 \choose k} x^k\big)}
{\sum_{k=0}^{\Delta - 2} \theta_k {\Delta - 1 \choose k} x^k} \nonumber
\\&\leq&
\theta_0^{-1}
\max_{k = 0,\ldots,\Delta - 2} \frac{(k+1) \theta_{k+1} {\Delta - 1 \choose k + 1}}{\theta_k {\Delta - 1 \choose k}}
\ \ \ =\ \ \ 
\theta_0^{-1} \psi_{\mb\theta};\label{gderiv1}
\end{eqnarray}
where the final inequality follows from convexity. 
\\\indent We now bound $|\partial_x f(x)|$.  For all $x \geq 0$, it follows from (\ref{fgoodrep1}) that
$$|\partial_x f(x)| \leq (\Delta - 1) 
\frac{ 
\sum_{i=0}^{\Delta - 1} \sum_{j = 0}^{\Delta - 2} {\Delta - 1 \choose i} {\Delta - 2 \choose j} x^{i + j} \big| \theta_i \theta_{j + 2} - \theta_{i + 1} \theta_{j + 1} \big|
}
{
\big(\sum_{k=0}^{\Delta - 1} \theta_k {\Delta - 1 \choose k} x^k \big)^2
}
.$$
Combining with the fact that non-negativity implies
\begin{eqnarray*}
\Bigg(\sum_{k=0}^{\Delta - 1} \theta_k {\Delta - 1 \choose k} x^k \Bigg)^2
&=& \sum_{i = 0}^{\Delta - 1} \sum_{j = 0}^{\Delta - 1} \theta_i \theta_j {\Delta - 1 \choose i} {\Delta - 1 \choose j} x^{i + j}
\\&\geq& \sum_{i = 0}^{\Delta - 1} \sum_{j = 0}^{\Delta - 2} \theta_i \theta_j {\Delta - 1 \choose i} {\Delta - 1 \choose j} x^{i + j},
\end{eqnarray*}
we conclude that
\begin{eqnarray*}
|\partial_x f(x)| &\leq& (\Delta - 1) 
\frac{ 
\sum_{i=0}^{\Delta - 1} \sum_{j = 0}^{\Delta - 2} {\Delta - 1 \choose i} {\Delta - 2 \choose j}  \big| \theta_i \theta_{j + 2} - \theta_{i + 1} \theta_{j + 1} \big| x^{i + j}
}
{
\sum_{i = 0}^{\Delta - 1} \sum_{j = 0}^{\Delta - 2} \theta_i \theta_j {\Delta - 1 \choose i} {\Delta - 1 \choose j} x^{i + j}
}
\\&=&
(\Delta - 1) 
\frac
{ 
\sum_{i=0}^{\Delta - 1} \sum_{j = 0}^{\Delta - 2} 
\frac{
{\Delta - 1 \choose i} {\Delta - 2 \choose j} \big| \theta_i \theta_{j + 2} - \theta_{i + 1} \theta_{j + 1} \big| x^{i + j}
}
{
\theta_i \theta_j {\Delta - 1 \choose i} {\Delta - 1 \choose j} x^{i + j}
}
\big( \theta_i \theta_j {\Delta - 1 \choose i} {\Delta - 1 \choose j} x^{i + j} \big)
}
{
\sum_{i = 0}^{\Delta - 1} \sum_{j = 0}^{\Delta - 2} \theta_i \theta_j {\Delta - 1 \choose i} {\Delta - 1 \choose j} x^{i + j}
}
\\&\leq&
(\Delta - 1) 
\max_{
\substack{i \in [0,\Delta-1]\\j \in [0,\Delta - 2]}} 
\frac{
{\Delta - 1 \choose i} {\Delta - 2 \choose j} \big| \theta_i \theta_{j + 2} - \theta_{i + 1} \theta_{j + 1} \big| x^{i + j}
}
{
\theta_i \theta_j {\Delta - 1 \choose i} {\Delta - 1 \choose j} x^{i + j}
}
\\&=&
\max_{
\substack{i \in [0,\Delta-1]\\j \in [0,\Delta - 2]}} 
\bigg( \big(\Delta - (j + 1) \big)\big| \frac{\theta_{j + 2}}{\theta_j} - \frac{\theta_{i + 1}}{\theta_i} \frac{\theta_{j+1}}{\theta_j}\big| \bigg),
\end{eqnarray*}
where the final inequality follows from convexity.  Further noting that the definition of $\psi_{\mb\theta}$ and log-convexity together imply that for all $i \in [0, \Delta  - 1]$ and $j \in [0, \Delta - 2]$,
\begin{eqnarray*}
\big(\Delta - (j + 1) \big)\big|\frac{\theta_{j + 2}}{\theta_j} - \frac{\theta_{i + 1}}{\theta_i} \frac{\theta_{j+1}}{\theta_j}\big|
&=&
\big(\Delta - (j + 1) \big)\frac{\theta_{j+1}}{\theta_j}\big|\frac{\theta_{j + 2}}{\theta_{j+1}} - \frac{\theta_{i + 1}}{\theta_i}\big|
\\&\leq&
\psi_{\mb\theta}(\frac{\theta_{\Delta}}{\theta_{\Delta-1}} - \frac{\theta_1}{\theta_0}),
\end{eqnarray*}
we may combine the above with the chain rule and Lemma\ \ref{monlemma} to conclude that for all $x \geq 0$,
\begin{equation}\label{fderiv2}
|\partial_x f^{\Delta - 1}(x)| \leq (\Delta - 1) (\frac{\theta_{\Delta}}{\theta_{\Delta-1}})^{\Delta - 2} (\frac{\theta_{\Delta}}{\theta_{\Delta-1}} - \frac{\theta_1}{\theta_0}) \psi_{\mb\theta}.
\end{equation}
It follows from (\ref{gderiv1}), (\ref{fderiv2}), Lemmas\ \ref{lemma3} and\ \ref{monlemma}, the fact that $|a b - c d| \leq |a + c| |b - d| + |b + d| |a - c|$ for all $a,b,c,d \in \reals$, that for all $d \geq 2$,
\begin{eqnarray*}
|\zeta_{d+1} - \zeta_d| &=& \big|\lambda g(\zeta_d) f^{\Delta - 1}(\zeta_{d-1}) - \lambda g(\zeta_{d-1}) f^{\Delta - 1}(\zeta_{d-2})\big|
\\&\leq& \lambda \big|g(\zeta_d) + g(\zeta_{d-1})\big| \big| f^{\Delta - 1}(\zeta_{d-1}) - f^{\Delta - 1}(\zeta_{d-2}) \big| 
\\&\ &\ \ \ +\ \ \ \lambda\big|f^{\Delta - 1}(\zeta_{d-1}) + f^{\Delta - 1}(\zeta_{d-2}) \big| \big|g(\zeta_d) - g(\zeta_{d-1})\big| 
\\&\leq& \lambda \big(2 \theta^{-1}_0\big) \bigg(|\zeta_{d-1} - \zeta_{d-2}| (\Delta - 1) (\frac{\theta_{\Delta}}{\theta_{\Delta-1}})^{\Delta - 2}
(\frac{\theta_{\Delta}}{\theta_{\Delta-1}} - \frac{\theta_1}{\theta_0}) \psi_{\mb\theta} \bigg)
\\&\ &\ \ \ +\ \ \ \lambda \big(2 (\frac{\theta_{\Delta}}{\theta_{\Delta-1}})^{\Delta - 1}\big)\bigg(|\zeta_d - \zeta_{d-1}| \theta_0^{-1} \psi_{\mb\theta} \bigg)
\end{eqnarray*}
Note that Lemma\ \ref{theorem2}, combined with our proof of Theorem\ \ref{impliesdecay1} (in particular the fact that $\sup_{d \geq 0} \zeta_{2 d} \leq \underline{x} \leq \overline{y} \leq \inf_{d \geq 0} \zeta_{2 d + 1}$), implies that $\lbrace |\zeta_{d+1} - \zeta_d|, d \geq 0 \rbrace$ is monotone decreasing.  Combining the above, we conclude that 
$$
|\zeta_{d+1} - \zeta_d| \leq 2 \lambda \theta^{-1}_0 (\frac{\theta_{\Delta}}{\theta_{\Delta-1}})^{\Delta - 2} \psi_{\mb\theta} \bigg(
(\Delta - 1)(\frac{\theta_{\Delta}}{\theta_{\Delta-1}} - \frac{\theta_1}{\theta_0}) + \frac{\theta_{\Delta}}{\theta_{\Delta-1}} \bigg) |\zeta_{d-1} - \zeta_{d-2}|.$$
It thus follows from the definition of $\underline{\lambda}_{\mb\theta}$ that for all $\lambda < \underline{\lambda}_{\mb\theta}$, there exists $\rho \in (0,1)$ such that for all $d \geq 2$,
$|\zeta_{d+1} - \zeta_d| \leq \rho |\zeta_{d-1} - \zeta_{d-2}|$.  Combining with our proof of Theorem\ \ref{impliesdecay1} (in particular the fact that existence of $\lim_{d \rightarrow \infty} \zeta_d$ implies uniqueness) completes the proof.
\\\\\\\indent We now prove that $(\lambda,\mb\theta)$ does not belong to the uniqueness regime for all $\lambda > \overline{\lambda}_{\mb\theta}$.  We first show that for $\lambda = \overline{\lambda}_{\mb\theta}$, any non-negative solution to the system of equations (\ref{limitequation1a})-(\ref{limitequation2a}) of the form $(x, x)$ satisfies 
\begin{equation}\label{verify1}
x \geq \frac{3}{\Delta} \frac{\theta_0}{\theta_1}.
\end{equation}
Indeed, it follows from log-convexity that $\theta_k \leq \theta_0 (\frac{\theta_{\Delta}}{\theta_{\Delta-1}})^k, k = 0,\ldots,\Delta$.  Thus for all $x \geq 0$,
$$
\sum_{k=0}^{\Delta - 1} \theta_k {\Delta - 1 \choose k} x^k\ \ \ \leq\ \ \ \theta_0 \sum_{k=0}^{\Delta - 1} {\Delta - 1 \choose k} (\frac{\theta_{\Delta}}{\theta_{\Delta-1}} x )^k\ \ \ =\ \ \ \theta_0\big(1 + \frac{\theta_{\Delta}}{\theta_{\Delta-1}} x \big)^{\Delta - 1},$$
and

$$
g(x)\ \ \ \geq\ \ \ \bigg(\theta_0 \big(1 + \frac{\theta_{\Delta}}{\theta_{\Delta-1}} x \big)^{\Delta-1}\bigg)^{-1}\ \ \ \geq\ \ \ \theta^{-1}_0 \exp\big( - \Delta  \frac{\theta_{\Delta}}{\theta_{\Delta-1}} x \big).
$$
Thus by Lemma\ \ref{monlemma}, any non-negative solution to the system of equations (\ref{limitequation1a})-(\ref{limitequation2a}) of the form $(x, x)$ for $\lambda = \overline{\lambda}_{\mb\theta}$ satisfies
\begin{equation}\label{gineqbb}
x \exp\big( \Delta \frac{\theta_{\Delta}}{\theta_{\Delta-1}} x \big)\ \ \ \geq\ \ \ \frac{\overline{\lambda}_{\mb\theta}}{\theta_0} (\frac{\theta_1}{\theta_0})^{\Delta - 1}\ \ \ =\ \ \ \frac{3}{\Delta} \frac{\theta_0}{\theta_1} \exp\big( 3 \frac{\theta_{\Delta}}{\theta_{\Delta-1}} \frac{\theta_0}{\theta_1} \big).
\end{equation}
To complete the proof of (\ref{verify1}), we observe that in light of (\ref{gineqbb}), $x < \frac{3}{\Delta} \frac{\theta_0}{\theta_1}$ would yield a contradiction, since it would imply 
$$x \exp\big( \Delta \frac{\theta_{\Delta}}{\theta_{\Delta-1}} x \big) < \frac{3}{\Delta} \frac{\theta_0}{\theta_1} \exp\big( 3 \frac{\theta_{\Delta}}{\theta_{\Delta-1}} \frac{\theta_0}{\theta_1} \big).$$
\ \\\indent We next prove that for any $x \geq \frac{3}{\Delta} \frac{\theta_0}{\theta_1}$ and $M \geq 1$, 
\begin{equation}\label{etaclaim}
\eta(M,x) \stackrel{\Delta}{=} M g(M x) \leq g(x),
\end{equation}
 a property that will allow us to use Lemma\ \ref{lemma3} to explicitly demonstrate that $\lbrace \zeta_d, d \geq 0 \rbrace$ exhibits a parity dependence.  We proceed by showing that for any $x \geq \frac{3}{\Delta} \frac{\theta_0}{\theta_1}$ and $M \geq 1$, $\partial_M \eta(M,x) \leq 0$.  Since $\partial_x g(x) = - g^2(x)  \sum_{k=1}^{\Delta-1} {\Delta - 1 \choose k} \theta_k k x^{k-1}$, it follows from the chain rule that
\begin{eqnarray*}
\partial_M \eta(M,x) &=& - M x g^2(M x) \sum_{k=1}^{\Delta-1} {\Delta - 1 \choose k} \theta_k k (M x)^{k-1} + g(M x)
\\&=& g(M x) \bigg( 1 - \frac{ \sum_{k=1}^{\Delta-1} {\Delta - 1 \choose k} \theta_k k (M x)^k}{\sum_{k=0}^{\Delta-1} {\Delta - 1 \choose k} \theta_k (M x)^k} \bigg).
\end{eqnarray*}
By the non-negativity of $g$, it thus suffices to demonstrate that for any $x \geq \frac{3}{\Delta} \frac{\theta_0}{\theta_1}$ and $M \geq 1$,
\begin{equation}\label{shownonu}
\sum_{k=1}^{\Delta-1} {\Delta - 1 \choose k} \theta_k k (M x)^k - \sum_{k=0}^{\Delta-1} {\Delta - 1 \choose k} \theta_k (M x)^k
\end{equation}
is non-negative.  Note that (\ref{shownonu}) equals 
$$\sum_{k = 1}^{\Delta - 1} {\Delta - 1 \choose k} \theta_k (k - 1) (M x)^k - \theta_0,$$
which by non-negativity and the fact that $M \geq 1$ is at least ${\Delta - 1 \choose 2} \theta_2 x^2 - \theta_0$.
As log-convexity implies $\theta_2 \geq \theta_0 (\frac{\theta_1}{\theta_0})^2$, and it is easily verified that ${\Delta - 1 \choose 2} \geq \frac{\Delta^2}{9}$ for all $\Delta \geq 3$, we conclude that $x \geq \frac{3}{\Delta} \frac{\theta_0}{\theta_1}$ implies that (\ref{shownonu}) is at least $\frac{\Delta^2}{9} \theta_0 (\frac{\theta_1}{\theta_0})^2 \big(\frac{3}{\Delta} \frac{\theta_0}{\theta_1}\big)^2 - \theta_0 = 0$, completing the proof.
\\\\\indent We now use (\ref{verify1}) and (\ref{etaclaim}) to prove by induction that for $\lambda > \overline{\lambda}_{\mb\theta}$, $\lbrace \zeta_d, d \geq 0 \rbrace$ exhibits a parity dependence, mirroring our proof of Theorem\ \ref{impliesdecay1}.
Recall from our proof of Theorem\ \ref{impliesdecay1} that for $\lambda = \overline{\lambda}_{\mb\theta}$, the system of equations 
(\ref{limitequation1a}) - (\ref{limitequation2a}) always has at least one non-negative solution of the form $(x, x)$.  Let us fix any such solution $(x_{\mb\theta},x_{\mb\theta})$, and note that $x_{\mb\theta} \geq \frac{3}{\Delta} \frac{\theta_0}{\theta_1}$.  We now prove by induction that for all $\lambda > \overline{\lambda}_{\mb\theta}$, $\lbrace \zeta_d, d \geq 0 \rbrace$ has a non-vanishing parity-dependence, with even values lying below $x_{\mb\theta}$, and odd values lying above $\frac{\lambda}{\overline{\lambda}_{\mb\theta}} x_{\mb\theta}$.  The $d = 0$ base case follows from non-negativity.  For the case $d = 1$, recall from definitions that $\zeta_1 = \lambda \theta^{-1}_0 (\frac{\theta_{\Delta}}{\theta_{\Delta - 1}})^{\Delta - 1}$.  However, by virtue of satisfying (\ref{limitequation1a}) - (\ref{limitequation2a}) with $\overline{\lambda}_{\mb\theta}$ and Lemma\ \ref{monlemma}, we have $x_{\mb\theta} \leq \overline{\lambda}_{\mb\theta} \theta^{-1}_0 (\frac{\theta_{\Delta}}{\theta_{\Delta - 1}})^{\Delta - 1}$.  Combining the above completes the proof for the $d = 1$ case.  Now, proceeding by induction, suppose that for some $d \geq 1$ and all $k \in \lbrace 0,\ldots, d - 1 \rbrace$, $\zeta_{2k} \leq x_{\mb\theta}$, and $\zeta_{2k + 1} \geq \frac{\lambda}{\overline{\lambda}_{\mb\theta}} x_{\mb\theta}$.  Then it follows from Lemma\ \ref{lemma3}, the monotonicity of $f$ and $g$ implied by Lemma\ \ref{monlemma}, (\ref{verify1}), and (\ref{etaclaim}) that
\begin{eqnarray*}
\zeta_{2d} &=& \lambda g(\zeta_{2d - 1}) f^{\Delta - 1}(\zeta_{2d - 2})
\\&\leq& \lambda g(\frac{\lambda}{\overline{\lambda}_{\mb\theta}} x_{\mb\theta}) f^{\Delta - 1}(x_{\mb\theta})
\\&\leq& \lambda \frac{\overline{\lambda}_{\mb\theta}}{\lambda} g(x_{\mb\theta}) f^{\Delta - 1}(x_{\mb\theta})\ \ \ =\ \ \ \overline{\lambda}_{\mb\theta} g(x_{\mb\theta}) f^{\Delta - 1}(x_{\mb\theta})\ \ \ =\ \ \ x_{\mb\theta},
\end{eqnarray*}
with the final inequality following from (\ref{etaclaim}).  Similarly,
\begin{eqnarray*}
\zeta_{2d + 1} &=& \lambda g(\zeta_{2d}) f^{\Delta - 1}(\zeta_{2d - 1})
\\&\geq& \lambda g(x_{\mb\theta}) f^{\Delta - 1}(\frac{\lambda}{\overline{\lambda}_{\mb\theta}} x_{\mb\theta}) 
\\&=& \frac{\lambda}{\overline{\lambda}_{\mb\theta}} \overline{\lambda}_{\mb\theta}  g(x_{\mb\theta}) f^{\Delta - 1}(\frac{\lambda}{\overline{\lambda}_{\mb\theta}} x_{\mb\theta})
\ \ \ \geq\ \ \ \frac{\lambda}{\overline{\lambda}_{\mb\theta}} \overline{\lambda}_{\mb\theta}  g(x_{\mb\theta}) f^{\Delta - 1}(x_{\mb\theta})\ \ \ =\ \ \ \frac{\lambda}{\overline{\lambda}_{\mb\theta}} x_{\mb\theta},
\end{eqnarray*}
with the final inequality following from the monotonicity of $f$ and fact that $\frac{\lambda}{\overline{\lambda}_{\mb\theta}} > 1$.  This completes the desired induction, demonstrating that $\lbrace \zeta_d, d \geq 0 \rbrace$ exhibits the stated parity-dependence.  
Combining with our proof of Theorem\ \ref{impliesdecay1} (in particular the fact that non-existence of $\lim_{d \rightarrow \infty} \zeta_d$ implies non-uniqueness) completes the proof.
\end{proof}
\section{A perturbative analysis, and proof of Theorem\ \ref{whenunique}}\label{Sec3}
In this section, we perform a perturbative analysis of the system of equations arising from our necessary and sufficient conditions for uniqueness, proving Theorem\ \ref{whenunique}.  First, it will be useful to rewrite the system of equations (\ref{limitequation1a}) - (\ref{limitequation2a}), which will allow us to apply known results from the theory of dynamical systems.  
Note that if $p_{\mb\theta}(x) \stackrel{\Delta}{=} x f_{\mb\theta}^{-(\Delta-1)}(x)$ is strictly increasing on $\big[0, \lambda 
(\frac{\theta_{\Delta}}{\theta_{\Delta-1}})^{\Delta - 1} \theta_0^{-1}\big]$, then it follows from Lemma\ \ref{monlemma} that $p_{\mb\theta}$ has a well-defined and unique inverse $p_{\mb\theta}^{\leftarrow}$, with domain a superset of $[0, \lambda \theta_0^{-1}]$ and range a subset of $\reals^+$, i.e. $p_{\mb\theta}^{\leftarrow}\big( p_{\mb\theta}(x) \big) = x$.  In this case we can define $q_{\lambda,\mb\theta}(x) \stackrel{\Delta}{=} p_{\mb\theta}^{\leftarrow}\big( \lambda g_{\mb\theta}(x) \big)$, and we observe that the system of equations (\ref{limitequation1a}) - (\ref{limitequation2a}) may be rewritten as follows.
\begin{observation}\label{rewritesoe}
If $\mb\theta$ is log-convex, and $p_{\mb\theta}(x)$ is strictly increasing on $\big[ 0 , \lambda (\frac{\theta_{\Delta}}{\theta_{\Delta-1}})^{\Delta - 1} \theta_0^{-1} \big]$, then on $\reals^+ \times \reals^+$, the system of equations (\ref{limitequation1a}) - (\ref{limitequation2a}) is equivalent to the system of equations
\begin{equation}\label{limitequation1aaa}
q_{\lambda,\mb\theta}\big(q_{\lambda,\mb\theta}(x)\big)\ =\ x,
\end{equation}
\begin{equation}\label{limitequation2aaa}
y =\ q_{\lambda,\mb\theta}(x).
\end{equation}
Furthermore, $q_{\lambda,\mb\theta}$ is strictly decreasing on $\reals^+$, and the equation $q_{\lambda,\mb\theta}(x) = x$ has a unique solution $x_{\lambda,\mb\theta}$ on $\reals^+$.  Also, it follows from Lemma\ \ref{monlemma} that every solution $(x,y)$ to the system of equations (\ref{limitequation1aaa}) - (\ref{limitequation2aaa}) on $\reals^+ \times \reals^+$ satisfies $0 \leq x,y \leq \lambda (\frac{\theta_{\Delta}}{\theta_{\Delta-1}})^{\Delta-1} \theta_0^{-1}$.  In addition, $x \in [0,
\lambda (\frac{\theta_{\Delta}}{\theta_{\Delta-1}})^{\Delta-1} \theta_0^{-1}]$ implies $q_{\lambda,\mb\theta}(x) \in [0,
\lambda (\frac{\theta_{\Delta}}{\theta_{\Delta-1}})^{\Delta-1} \theta_0^{-1}]$.
\end{observation}
It is well-known from the theory of dynamical systems that under certain additional assumptions on $q_{\lambda,\mb\theta}$, necessary and sufficient conditions for when the system of equations (\ref{limitequation1aaa}) - (\ref{limitequation2aaa}) has a unique solution can be stated in terms of whether the map $q_{\lambda,\mb\theta}$ exhibits a certain local stability at the fixed point $x_{\lambda,\mb\theta}$.  We now make this precise, and note that our approach is similar to that taken previously in the literature to analyze related models (cf. \cite{LRZ.06}).  Recall that for a thrice-differentiable function $F(x)$ with non-vanishing derivative on some interval $I$, we define (on $I$) the \emph{Schwarzian derivative} of $F$ as the function 
$$S[F] \stackrel{\Delta}{=} \frac{ \frac{d^3}{dx^3} F }{\frac{d}{dx} F} - \frac{3}{2} \big( \frac{ \frac{d^2}{dx^2} F }{\frac{d}{dx} F} \big)^2.$$
For a function $F$ and $n \geq 1$, let $F^{\lbrace n \rbrace}(x)$ denote the $n$-fold iterate of $F$, i.e. $F^{\lbrace n+1 \rbrace}(x) = F\big(F^{\lbrace n \rbrace}(x)\big)$, with $F^{\lbrace 1 \rbrace}(x) = F(x)$.  Then the following well-known result from dynamical systems is stated in Lemma 4.3 of \cite{LRZ.06}.  
\begin{theorem}\label{dynamical1}
Suppose $I = [L,R] \subseteq {\mathcal R}$ is some closed bounded interval, and $F$ is some function with the following properties.
\begin{enumerate}[(i)]
\item $F$ has domain $I$, and range a subset of $I$.  \label{dynamo1}
\item The third derivative of $F$ exists and is continuous on $I$. \label{dynamo2}
\item The equation $x = F(x)$ has a unique solution $x^*$ on $I$. \label{dynamo3}
\item $F$ is a decreasing function on $I$. \label{dynamo4}
\item $S[F](x) < 0$ for all $x \in I$. \label{dynamo5}
\end{enumerate}
Then $\lim_{n \rightarrow \infty} F^{\lbrace n \rbrace}(x)$ exists and equals $x^*$ for all $x \in I$ iff $|\partial_x F(x^*)| \leq 1$ iff 
$\lim_{n \rightarrow \infty} F^{\lbrace n \rbrace}(L) = x^*$.
\end{theorem}
We now customize Theorem\ \ref{dynamical1} to our own setting.  Let $r_{\mb\theta}(x) \stackrel{\Delta}{=} |\frac{ \partial_x g_{\mb\theta}(x) }{\partial_x p_{\mb\theta}(x)}|$.  Then combined with Observation\ \ref{rewritesoe}, Theorem\ \ref{dynamical1} implies the following.
\begin{observation}\label{dynamicobs2}
Suppose that $\mb\theta$ is log-convex, $p_{\mb\theta}(x)$ is strictly increasing on $\big[ 0 , \lambda (\frac{\theta_{\Delta}}{\theta_{\Delta-1}})^{\Delta - 1} \theta_0^{-1} \big]$, and the conditions of Theorem\ \ref{dynamical1} are satisfied with $F = 
q_{\lambda,\mb\theta}, I = \big[ 0 , \lambda (\frac{\theta_{\Delta}}{\theta_{\Delta-1}})^{\Delta - 1} \theta_0^{-1} \big]$.  Then 
$(\lambda,\mb\theta)$  belongs to the uniqueness regime iff  $r_{\mb\theta}(x_{\lambda,\mb\theta}) \leq \lambda^{-1}$.
\end{observation}
\begin{proof}
We first prove that the system of equations (\ref{limitequation1aaa}) - (\ref{limitequation2aaa}) does not have a unique solution on $\reals^+ \times \reals^+$ iff $|\partial_x q_{\lambda,\mb\theta}(x_{\lambda,\mb\theta})| > 1$.  Suppose the system of equations (\ref{limitequation1aaa}) - (\ref{limitequation2aaa}) does not have a unique solution on $\reals^+ \times \reals^+$.  Since Observation\ \ref{rewritesoe} implies that the equation $q_{\lambda,\mb\theta}(x) = x$ has a unique solution $x_{\lambda,\mb\theta}$, it follows that there must exist a solution $(x,y)$ to the system of equations (\ref{limitequation1aaa}) - (\ref{limitequation2aaa}) with $x < y$.  In this case, $\lim_{n \rightarrow \infty} q_{\lambda,\mb\theta}^{\lbrace n \rbrace}(x)$ does not exist, as the series alternates between $x$ and $y$, and it follows from Theorem\ \ref{dynamical1} that 
$|\partial_x q_{\lambda,\mb\theta}(x_{\lambda,\mb\theta})| > 1$.
\\\indent Alternatively, suppose that $|\partial_x q_{\lambda,\mb\theta}(x_{\lambda,\mb\theta})| > 1$.  Then it follows from 
Theorem\ \ref{dynamical1} that $\lim_{n \rightarrow \infty} q_{\lambda,\mb\theta}^{\lbrace n \rbrace}(0)$ does not exist.  
However, both $Z_{\textrm{even}} \stackrel{\Delta}{=} \lim_{n \rightarrow \infty} q_{\lambda,\mb\theta}^{\lbrace 2 n \rbrace}(0)$ and $Z_{\textrm{odd}} \stackrel{\Delta}{=} \lim_{n \rightarrow \infty} q_{\lambda,\mb\theta}^{\lbrace 2 n + 1 \rbrace}(0)$ both exist.  Indeed, this follows from the fact that $q^{\lbrace 1 \rbrace}_{\lambda,\mb\theta}$ is decreasing, $q^{\lbrace 2 \rbrace}_{\lambda,\mb\theta}$ is increasing, $q^{\lbrace 2 \rbrace}_{\lambda,\mb\theta}(0) \geq 0$, and $q^{\lbrace 3 \rbrace}_{\lambda,\mb\theta}(0) \leq q^{\lbrace 1 \rbrace}_{\lambda,\mb\theta}(0)$, which implies that both relevant sequences are appropriately monotone.  Noting that the non-existence of the stated limit implies $Z_{\textrm{even}} \neq Z_{\textrm{odd}}$, and the pair $(Z_{\textrm{even}},Z_{\textrm{odd}})$ must be a solution to the system of equations (\ref{limitequation1aaa}) - (\ref{limitequation2aaa}), completes the desired demonstration.
\\\indent As it follows from elementary calculus that $\partial_x q_{\lambda,\mb\theta}(x_{\lambda,\mb\theta}) =   \lambda \frac{ \partial_x g_{\mb\theta}(x_{\lambda,\mb\theta}) }{ \partial_x p_{\mb\theta}(x_{\lambda,\mb\theta})}$, combining the above with Theorem\ \ref{impliesdecay1} and Observation\ \ref{rewritesoe} completes the proof.
\end{proof}
We note that $p_{\mb\theta}$ is not necessarily an increasing function for the case of general log-convex $\mb{\theta}$.  Furthermore, even when $p_{\mb\theta}$ is increasing, an analysis of $S[q_{\lambda,\mb\theta}]$ seems difficult, and the associated uniqueness regime of the parameter space seems to be quite complex.  However, for the special setting in which $\mb{\theta}$ belongs to a neighborhood of the all ones vector, in which case the associated M.r.f. becomes a perturbation of the hardcore model at criticality, these difficulties can be overcome by expanding the relevant functions using appropriate Taylor series.   The theory of real analytic functions provides a convenient framework for proving the validity of these expansions, and we refer the reader to \cite{krantz2002primer} for details.  Using this framework, we prove the following.
\begin{lemma}\label{hdynamo}
For each convex vector $\mathbf{c}$, and $U \in \reals^+$, there exists $\delta_{\mathbf{c},U} > 0$ such that the following hold.\begin{enumerate}[(i)]
\item  $g_{\mathbf{1} + \mathbf{c} h}(x)$ and $p_{\mathbf{1} + \mathbf{c} h}(x)$ are jointly real analytic functions of $(h,x)$ on $[0,\delta_{\mathbf{c},U}] \times [0,U]$.  For each fixed $h \in [0,\delta_{\mathbf{c},U}]$ and all $x \in [0,U]$, $\partial_x g_{\mathbf{1} + \mathbf{c} h}(x) < 0$, and $\partial_x p_{\mathbf{1} + \mathbf{c} h}(x) > 0$. \label{hdynamo0}
\item For each fixed $h \in [0,\delta_{\mathbf{c},U}]$, $p_{\mathbf{1} + \mathbf{c} h}(x)$ has a well-defined and unique inverse $p^{\leftarrow}_{\mathbf{1} + \mathbf{c} h}(x)$ with domain a superset of $[0,U]$ and range a subset of $\reals^+$.  Furthermore, $p^{\leftarrow}_{\mathbf{1} + \mathbf{c} h}(x)$ is a jointly real analytic function of $(h,x)$ on $[0,\delta_{\mathbf{c},U}] \times [0,U]$. \label{hdynamo1}
\item $q_{\lambda_{\Delta},\mathbf{1} + \mathbf{c} h}(x)$, $\partial_x q_{\lambda_{\Delta},\mathbf{1} + \mathbf{c} h}(x)$, and $S[q_{\lambda_{\Delta},\mathbf{1} + \mathbf{c} h}](x)$ are all jointly real analytic functions of $(h,x)$ on $[0,\delta_{\mathbf{c},U}] \times  [0,U]$.  Furthermore $\partial_x q_{\lambda_{\Delta},\mathbf{1} + \mathbf{c} h}(x)$ and $S[q_{\lambda_{\Delta},\mathbf{1} + \mathbf{c} h}](x)$ are strictly negative for all $(h,x) \in [0,\delta_{\mathbf{c},U}] \times  [0,U]$. \label{hdynamo2}
\end{enumerate}
\end{lemma}
\begin{proof}
We prove (\ref{hdynamo0}) - (\ref{hdynamo2}) in order.
\\\\ (\ref{hdynamo0}).  The claim with respect to real analyticity follows from the fact that for any fixed $U_1$, there exists $\delta_{1,U_1} > 0$ such that both $g_{\mathbf{1} + \mathbf{c} h}(x)$ and $p_{\mathbf{1} + \mathbf{c} h}(x)$ are ratios of non-vanishing polynomials of $(h,x)$ on $[0,\delta_{1,U_1}] \times [0,U_1]$.
That there exists $\delta_{2,U_1} > 0$ such that $\partial_x g_{\mathbf{1} + \mathbf{c} h}(x) < 0$ and $\partial_x p_{\mathbf{1} + \mathbf{c} h}(x) > 0$ for all  $(h,x) \in [0,\delta_{2,U_1}] \times [0,U_1]$ then follows from the fact that $\partial_x g_{\mathbf{1}}(x) = - (\Delta - 1) (x + 1)^{-\Delta}$, and 
$\partial_x p_{\mathbf{1}}(x) = 1$. 
\\\\ (\ref{hdynamo1}).  The claim follows from (\ref{hdynamo0}) and the inverse function theorem for real analytic functions (cf. \cite{krantz2002primer}).
\\\\(\ref{hdynamo2}).  The claim with respect to $q_{\lambda_{\Delta},\mathbf{1} + \mathbf{c} h}(x)$ and $\partial_x q_{\lambda_{\Delta},\mathbf{1} + \mathbf{c} h}(x)$ follows from (\ref{hdynamo1}), and the fact that $\partial_x q_{\lambda_{\Delta},\mathbf{1}}(x) = - (\Delta - 1) \lambda_{\Delta} (x + 1)^{-\Delta}$.  As this implies that $\partial_x q_{\lambda_{\Delta},\mathbf{1}}(x)$ is strictly negative (and thus non-vanishing), the desired claim with respect to $S[q_{\lambda_{\Delta},\mathbf{1} + \mathbf{c} h}](x)$ then follows from the fact that $S[q_{\lambda_{\Delta},\mathbf{1}}](x) = \frac{- \Delta (\Delta - 2)}{2 (x+1)^2}$.
\end{proof}
Combining Observations\ \ref{dynamicobs2}\ and\ \ref{limlogcon} with Lemma\ \ref{hdynamo} immediately yields necessary and sufficient conditions for uniqueness when $\mb\theta$ is a convex perturbation of $\mathbf{1}$.
\begin{corollary}\label{dynamocor}
For each convex vector $\mathbf{c}$, there exists $\delta_{\mathbf{c}} > 0$ such that the following hold for all $h \in [0,\delta_{\mathbf{c}}]$.
\begin{enumerate}[(i)]
\item $q_{\lambda_{\Delta},\mathbf{1} + \mathbf{c} h}(x) - x$ is strictly decreasing on $[0 , 2 \lambda_{\Delta}]$, and has a unique zero $x_{\lambda_{\Delta}, \mathbf{1} + \mathbf{c} h}$ on $[0 , 2 \lambda_{\Delta}]$. \label{dynamocora}
\item $(\lambda_{\Delta},\mathbf{1} + \mathbf{c} h)$ belongs to the uniqueness regime iff $r_{\mathbf{1} + \mathbf{c} h}(x_{\lambda_{\Delta},\mathbf{1} + \mathbf{c} h}) \leq \lambda^{-1}_{\Delta}$. \label{dynamocorb}
\end{enumerate}
\end{corollary}
With Corollary\ \ref{dynamocor} in hand, we now complete the proof of Theorem\ \ref{whenunique}.  For $l \in \lbrace 0,1 \rbrace$ and $\mathbf{c} = (c_0,\ldots,c_{\Delta})$, 
let 
$$f_{l,\mb\theta}(x) \stackrel{\Delta}{=} \sum_{i=0}^{\Delta - 1} \theta_{i + l} {\Delta  - 1 \choose i} x^i\ \ \ ,\ \ \  
z_{l,\mathbf{c}} \stackrel{\Delta}{=} \sum_{i=0}^{\Delta-1}  {\Delta - 1 \choose i } x_{\lambda_{\Delta},\mathbf{1}}^i c_{i + l}\ \ \ ,\ \ \ 
w_{l,\mathbf{c}} \stackrel{\Delta}{=} \sum_{i=0}^{\Delta-1} {\Delta - 1 \choose i } i x_{\lambda_{\Delta},\mathbf{1}}^{i-1} c_{i + l},$$
and
$$x_{\mathbf{c}} \stackrel{\Delta}{=} 
\frac{1}{2} \frac{(\Delta-2)^{\Delta-2}}{(\Delta-1)^{\Delta-1}} \big( (\Delta - 1) z_{1,\mathbf{c}} - \Delta z_{0,\mathbf{c}} \big).$$
Also, let $o(h)$ denote the family of functions $F(h)$ such that $\lim_{h \downarrow 0} h^{-1} F(h) = 0$.  With a slight abuse of notation, we will also let $o(h)$ refer to any particular function belonging to this family.  Finally, in simplifying certain expressions, we will use the following identities, which follow from a straightforward calculation (the details of which we omit).
\begin{lemma}\label{goodxis}
$$x_{\lambda_{\Delta},\mathbf{1}} = (\Delta - 2)^{-1}\ \ \ ,\ \ \ 
\sum_{i=0}^{\Delta} \Lambda_{\Delta,i} = (\frac{\Delta-1}{\Delta-2})^{\Delta}\ \ \ ,\ \ \ 
\sum_{i=0}^{\Delta} i \Lambda_{\Delta,i} = \Delta \frac{(\Delta-1)^{\Delta - 1}}{(\Delta-2)^{\Delta}}$$
$$\sum_{i=0}^{\Delta} i^2 \Lambda_{\Delta,i} = 2 \Delta \frac{(\Delta-1)^{\Delta-1}}{(\Delta-2)^{\Delta}}\ \ \ ,\ \ \ 
\sum_{i=0}^{\Delta} \pi_i = - (\frac{\Delta-1}{\Delta-2})^{\Delta - 1}\ \ \ ,\ \ \ 
\sum_{i=0}^{\Delta} i \pi_i =  \Delta (\frac{\Delta-1}{\Delta-2})^{\Delta - 1}.
$$
\end{lemma}
\begin{proof}[Proof of Theorem\ \ref{whenunique}]
We proceed by analyzing $r_{\mathbf{1} + \mathbf{c} h}(x_{\lambda_{\Delta},\mathbf{1} + \mathbf{c} h}) - \lambda^{-1}_{\Delta}$ as $h \downarrow 0$, and begin by proving that
\begin{equation}\label{whatisx}
\lim_{h \downarrow 0} (x_{\lambda_{\Delta},\mathbf{1} + \mathbf{c} h} - x_{\lambda_{\Delta},\mathbf{1}}) h^{-1} = x_{\mathbf{c}}.
\end{equation}
Note that for any fixed $\alpha \in \reals$ and $l \in \lbrace 0,1 \rbrace$,
\begin{eqnarray} 
f_{l,\mathbf{1} + \mathbf{c} h}(x_{\lambda_{\Delta},\mathbf{1}} + \alpha h) &=& \sum_{i=0}^{\Delta-1}  {\Delta - 1 \choose i } (1 + c_{i + l} h) \sum_{j=0}^i {i \choose j} x_{\lambda_{\Delta},\mathbf{1}}^{i-j} (\alpha h)^j\nonumber
\\&=& \sum_{i=0}^{\Delta-1}  {\Delta - 1 \choose i } x_{\lambda_{\Delta},\mathbf{1}}^i (1 + c_{i + l} h)( 1 + i x_{\lambda_{\Delta},\mathbf{1}}^{- 1} \alpha h) + o(h)\nonumber
\\&=& (1 + x_{\lambda_{\Delta},\mathbf{1}})^{\Delta - 1} + \big( (\Delta - 1)(1 + x_{\lambda_{\Delta},\mathbf{1}})^{\Delta - 2} \alpha + z_{l,\mathbf{c}} \big) h + o(h).\label{flhexp}
\end{eqnarray}
We conclude that
\begin{equation}
g_{\mathbf{1} + \mathbf{c} h}(x_{\lambda_{\Delta},\mathbf{1}} + \alpha h)
=
(1 + x_{\lambda_{\Delta},\mathbf{1}})^{-(\Delta - 1)} 
- 
(1 + x_{\lambda_{\Delta},\mathbf{1}})^{-2 (\Delta - 1)}
\big( (\Delta - 1)(1 + x_{\lambda_{\Delta},\mathbf{1}})^{\Delta - 2} \alpha + z_{0,\mathbf{c}} \big) h + o(h),\label{gexp}
\end{equation}
and
\begin{equation}
f_{\mathbf{1} + \mathbf{c} h}(x_{\lambda_{\Delta},\mathbf{1}} + \alpha h) 
=
1 + (1 + x_{\lambda_{\Delta},\mathbf{1}})^{-(\Delta - 1)}(z_{1,\mathbf{c}} - z_{0,\mathbf{c}}) h + o(h).\label{fkhexp}
\end{equation}
It follows from (\ref{gexp}), (\ref{fkhexp}), and a straightforward calculation (the details of which we omit) that for $\alpha \in \reals$,
$$
(x_{\lambda_{\Delta},\mathbf{1}} + \alpha h) - \lambda_{\Delta} f^{\Delta - 1}_{\mathbf{1} + \mathbf{c} h}(x_{\lambda_{\Delta},\mathbf{1}} + \alpha h) g_{\mathbf{1} + \mathbf{c} h}(x_{\lambda_{\Delta},\mathbf{1}} + \alpha h)
=
2 (\alpha -  x_{\mathbf{c}}) h + o(h).$$
Combining with Corollary\ \ref{dynamocor}.(\ref{dynamocora}), and the fact that $x_{\lambda_{\Delta},\mathbf{1}} < \lambda_{\Delta}$, completes the proof.
\\\indent Next, we use (\ref{whatisx}) to prove that
\begin{equation}\label{cleanitup1}
\partial_x p_{\mathbf{1} + \mathbf{c} h}(x_{\lambda_{\Delta},\mathbf{1} + \mathbf{c} h}) + \lambda_{\Delta} \partial_x g_{\mathbf{1} + \mathbf{c} h}(x_{\lambda_{\Delta},\mathbf{1} + \mathbf{c} h})
=
- \frac{1}{2}  (\frac{\Delta-2}{\Delta-1})^{\Delta} \mb{\pi} \cdot \mathbf{c} h + o(h).
\end{equation}
Indeed, it follows from (\ref{whatisx}) that
\begin{eqnarray} 
\partial_x f_{l,\mathbf{1} + \mathbf{c} h}(x_{\lambda_{\Delta},\mathbf{1} + \mathbf{c} h}) &=& 
\sum_{i=1}^{\Delta-1}  {\Delta - 1 \choose i } i (1 + c_{i + l} h) \sum_{j=0}^{i-1} {i - 1 \choose j} x_{\lambda_{\Delta},\mathbf{1}}^{i-1-j} (x_{\mathbf{c}} h)^j + o(h) \nonumber
\\&=& 
\sum_{i=1}^{\Delta-1}  {\Delta - 1 \choose i } i x_{\lambda_{\Delta},\mathbf{1}}^i (1 + c_{i + l} h) \big(x^{-1}_{\lambda_{\Delta},\mathbf{1}} + (i-1) x^{-2}_{\lambda_{\Delta},\mathbf{1}} x_{\mathbf{c}} h
\big) + o(h),\nonumber
\end{eqnarray}
which itself equals
\begin{equation}\label{dflhexp}
(\Delta - 1) (1 + x_{\lambda_{\Delta},\mathbf{1}})^{\Delta - 2} + \bigg( x_{\mathbf{c}} (\Delta - 1) (\Delta - 2) (1 + x_{\lambda_{\Delta},\mathbf{1}})^{\Delta - 3} + w_{l,\mathbf{c}} \bigg)h + o(h). 
\end{equation}
It follows from (\ref{whatisx}) - (\ref{dflhexp}), and a straightforward calculation (the details of which we omit), that 
$$\partial_x g_{\mathbf{1} + \mathbf{c} h}(x_{\lambda_{\Delta},\mathbf{1} + \mathbf{c} h}) 
= - g^{2}_{\mathbf{1} + \mathbf{c} h}(x_{\lambda_{\Delta},\mathbf{1} + \mathbf{c} h}) \partial_x f_{0,\mathbf{1} + \mathbf{c} h}(x_{\lambda_{\Delta},\mathbf{1} + \mathbf{c} h}),$$ which itself equals 
\begin{equation}\label{dghexp}
 - (\Delta - 1) (1 + x_{\lambda_{\Delta},\mathbf{1}})^{-\Delta} + (1 + x_{\lambda_{\Delta},\mathbf{1}})^{-(2 \Delta - 1)} 
\bigg( - (1 + x_{\lambda_{\Delta},\mathbf{1}}) w_{0,\mathbf{c}}
+ \Delta (\Delta - 1) (1 + x_{\lambda_{\Delta},\mathbf{1}})^{\Delta - 2} x_{\mathbf{c}}
+ 2 (\Delta - 1) z_{0,\mathbf{c}} \bigg)h + o(h);
\end{equation}
$\partial_x f_{\mathbf{1} + \mathbf{c} h}(x_{\lambda_{\Delta},\mathbf{1} + \mathbf{c} h})$ equals 
$$g^{2}_{\mathbf{1} + \mathbf{c} h}(x_{\lambda_{\Delta},\mathbf{1} + \mathbf{c} h}) \bigg(
f_{0,\mathbf{1} + \mathbf{c} h}(x_{\lambda_{\Delta},\mathbf{1} + \mathbf{c} h}) \partial_x f_{1,\mathbf{1} + \mathbf{c} h}(x_{\lambda_{\Delta},\mathbf{1} + \mathbf{c} h})
- f_{1,\mathbf{1} + \mathbf{c} h}(x_{\lambda_{\Delta},\mathbf{1} + \mathbf{c} h}) \partial_x f_{0,\mathbf{1} + \mathbf{c} h}(x_{\lambda_{\Delta},\mathbf{1} + \mathbf{c} h})\bigg),$$ which itself equals
$$
(1 + x_{\lambda_{\Delta},\mathbf{1}})^{-\Delta}\bigg( (1 + x_{\lambda_{\Delta},\mathbf{1}})(w_{1,\mathbf{c}} - w_{0,\mathbf{c}}) + (\Delta - 1)
(z_{0,\mathbf{c}} - z_{1,\mathbf{c}}) \bigg) h + o(h);
$$
and
$$\partial_x p_{\mathbf{1} + \mathbf{c} h}(x_{\lambda_{\Delta},\mathbf{1} + \mathbf{c} h})
= f^{-(\Delta - 1)}_{\mathbf{1} + \mathbf{c} h}(x_{\lambda_{\Delta},\mathbf{1} + \mathbf{c} h})  
- (\Delta - 1) x_{\lambda_{\Delta},\mathbf{1} + \mathbf{c} h}  f^{-\Delta}_{\mathbf{1} + \mathbf{c} h}(x_{\lambda_{\Delta},\mathbf{1} + \mathbf{c} h}) \partial_x f_{\mathbf{1} + \mathbf{c} h}(x_{\lambda_{\Delta},\mathbf{1} + \mathbf{c} h}),$$ which itself equals
\begin{equation}\label{dphexp}
1 - (\Delta-1) x_{\lambda_{\Delta},\mathbf{1}} (1 + x_{\lambda_{\Delta},\mathbf{1}})^{-(\Delta-1)} (w_{1,\mathbf{c}}-w_{0,\mathbf{c}}) h + o(h). 
\end{equation}
Combining (\ref{dghexp}) - (\ref{dphexp}) with Lemma\ \ref{goodxis} and simplifying, we conclude that the left-hand side of 
(\ref{cleanitup1}) equals 
\begin{equation}\label{finaltopi}
\frac{(\Delta-2)^{\Delta-2}}{2(\Delta-1)^{\Delta}}
\bigg(
-(\Delta-2)^3 z_{0,\mathbf{c}} + \Delta(\Delta-1)(\Delta-2) z_{1,\mathbf{c}}
+ 2 (\Delta - 1)(\Delta - 2) w_{0,\mathbf{c}} - 2 (\Delta-1)^2 w_{1,\mathbf{c}} \bigg) h + o(h).
\end{equation}
It follows from the definition of $\mb{\pi}$ and a further straightforward algebraic manipulation that (\ref{finaltopi}) equals 
$- \frac{1}{2}  (\frac{\Delta-2}{\Delta-1})^{\Delta} \mb{\pi} \cdot \mathbf{c} h + o(h)$, completing the desired demonstration.
\\\indent Combining (\ref{whatisx}), (\ref{cleanitup1}), and Corollary\ \ref{dynamocor} with the fact that 
$r_{\mathbf{1} + \mathbf{c} h}(x_{\lambda_{\Delta},\mathbf{1} + \mathbf{c} h}) \leq \lambda^{-1}_{\Delta}$
iff the l.h.s. of (\ref{cleanitup1}) is non-negative completes the proof.
\end{proof}
\section{Conclusion}\label{Concsec}
In this paper, we investigated second-order M.r.f. for independent sets on the infinite Cayley tree, a generalization of the hardcore model which arises in statistical physics, combinatorial optimization, and operations research, with an eye towards understanding which distributions can be attained for the number of included neighbors of an excluded node, while staying in the uniqueness regime.  We proved that the associated Gibbsian specification satisfies the FKG Inequality whenever the local potentials defining the Hamiltonian satisfy a certain log-convexity condition, which leads to so-called reverse ultra log-concave distributions for the number of included neighbors of an excluded node.  Under this condition, we gave necessary and sufficient conditions for the existence of a unique infinite-volume Gibbs measure in terms of an explicit system of equations, proved the existence of a phase transition, and gave explicit lower and upper bounds on the associated critical activity, which we proved to exhibit a certain robustness.  For potentials which are small perturbations of those coinciding to the hardcore model at its critical activity, we performed a perturbative analysis of the system of equations arising from our necessary and sufficient conditions for uniqueness, allowing us to explicitly characterize whether the resulting specification has a unique infinite-volume Gibbs measure in terms of whether these perturbations satisfy an explicit linear inequality.  Our analysis revealed an interesting non-monotonicity with regards to biasing towards excluded nodes with no included neighbors, which we used (in conjunction with our lower and upper bounds) to compare the uniqueness regime for our model to a related model in which the associated potentials have a simple factorized form.
\\\indent This work leaves many interesting directions for future research.  The full power of higher-order M.r.f. for sampling from independent sets in sparse graphs, and the associated  uniqueness regime, remains poorly understood.  Several questions build immediately on the models considered in this paper, such as developing a deeper understanding of the uniqueness regime for second-order M.r.f. with log-convex potentials, and more generally higher-order M.r.f. which also satisfy the FKG Inequality.  Analyzing settings in which the FKG Inequality no longer holds (at least for the partial order considered in this paper), e.g. second-order M.r.f. with log-concave potentials (which includes the restriction to maximal independent sets, and for which $f_{\mb\theta}$ is monotone decreasing instead of increasing), remains an open challenge.  It is also an open question to understand which sets of occupancy probabilities can be acheived by higher-order M.r.f. (in the uniqueness regime).  Can one use higher-order M.r.f. (in the uniqueness regime) to sample from denser independent sets than can be attained using the hardcore model at the critical activity?   It would also be interesting to study higher-order M.r.f. for related combinatorial problems, e.g. graph coloring, as well as for more general sparse graphs.  Indeed, although our results can be easily extended to the setting of regular graphs of large girth using standard techniques, proving results for general bounded-degree graphs (as was done by Weitz for the hardcore model in \cite{W.06}) seems to require fundamentally new ideas.  Similarly, the algorithmic implications of phase transitions for higher-order M.r.f. also remain open questions.  In particular, one would expect a ``complexity transition" at the uniqueness threshold with respect to approximately computing the relevant partition functions, as has been recently established for first-order M.r.f. (cf. \cite{SS.12}).  
\\\indent Finally, it is open to investigate the connection between higher-order M.r.f. and research on bernoulli shifts, i.i.d. factors of graphs, and local algorithms (cf. \cite{van1999existence,backhausz2014ramanujan}), which have played a prominent role recently in developing algorithms for finding dense independent sets in sparse graphs (cf. \cite{gamarnik2014limits,csoka2014invariant}).  For example, it has been proven that under certain additional technical assumptions, certain M.r.f. can (not) be well-approximated (in an appropriate sense) by i.i.d. factors of graphs (cf. \cite{van1999existence,backhausz2014large}).  However, a complete understanding of this and related questions seems beyond the reach of current techniques.  The converse, i.e. questions regarding whether an i.i.d. factor of graphs can be well-approximated by a finite-order M.r.f. in the uniqueness regime, seem to have received less attention in the literature, beyond the special case in which the underlying graph is a line and the M.r.f. reduces to a Markov chain (cf. \cite{ornstein1973application,rudolph1977limits}).  Such a connection could open the door to, e.g., searching the space of M.r.f. (which are easily parametrized on sparse graphs) to find specifications which sample from dense independent sets while remaining in the uniqueness regime.  We conclude with the following related question.
\begin{question}
To what extent are higher-order M.r.f. in the uniqueness regime capable of (approximately) encoding those distributions on independent sets which exhibit long-range independence, such as i.i.d. factors of graphs?
\end{question}
\section*{Acknowledgements} The author would like to thank Rob van den Berg, David Gamarnik, Kavita Ramanan, Justin Salez, and Prasad Tetali for several stimulating discussions.  The author especially thanks Rob van den Berg for a very helpful conversation regarding the applicability of the FKG Theorem.
\bibliographystyle{amsplain}
\bibliography{Gibbs_bib_5_20_2015}
\end{document}